\documentclass[a4paper,12pt]{article}

\usepackage[affil-it]{authblk}

\title{On 2D Harmonic Extensions of Vector Fields and Stellarator Coils}
\author{Adam J. Golab, James C. Robinson, Jos\'e L. Rodrigo}
\affil{Mathematics Research Centre, Zeeman Building, \\ University of Warwick, Coventry CV4 7AL, \\ United Kingdom}
\date{}

\usepackage[margin=1in]{geometry}

\usepackage{amsmath}
\usepackage{amssymb}
\usepackage{amsfonts}
\usepackage{amsthm}
\usepackage{enumerate}
\usepackage{empheq}
\usepackage{bm}
\usepackage{cases}
\usepackage{hyperref}
\usepackage{xcolor}

\usepackage{float}
\usepackage{tikz}
\usetikzlibrary{calc}
\usetikzlibrary{intersections}
\usetikzlibrary{fit}
\usetikzlibrary{through}
\usetikzlibrary{patterns}

\newtheorem{thm}{Theorem}[section] 
\newtheorem{lem}[thm]{Lemma}

\theoremstyle{definition}

\newtheorem{eg}[thm]{Example}
\newtheorem{rmk}[thm]{Remark}

\newcommand{\B}{\bm{B}}

\newcommand{\T}{\bm{t}}
\newcommand{\N}{\bm{n}}

\newcommand{\RE}{\operatorname{Re}}
\newcommand{\IM}{\operatorname{Im}}

\begin{document}
\maketitle

\begin{abstract}
\let\thefootnote\relax\footnotetext{E-mail addresses: \texttt{A.J.Golab@warwick.ac.uk}, \texttt{J.C.Robinson@warwick.ac.uk}, \\ \texttt{J.Rodrigo@warwick.ac.uk}.}
We consider a problem relating to magnetic confinement devices known as stellarators. Plasma is confined by magnetic fields generated by current-carrying coils, and here we investigate how closely to the plasma they need to be positioned. Current-carrying coils are represented as singularities within the magnetic field and therefore this problem can be modelled mathematically as finding how far we can harmonically extend a vector field from the boundary of a domain.

For this paper we consider two-dimensional domains with real analytic boundary, and prove that a harmonic extension exists if and only if the boundary data satisfies a combined compatibility and regularity condition. Our method of proof uses a generalisation of a result of Hadamard on the Cauchy problem for the Laplacian. We then provide a lower bound on how far we can harmonically extend the vector field from the boundary via the Cauchy--Kovalevskaya Theorem. 
\end{abstract}

\section{Introduction}

The motivation for the results in this paper arises from the study of magnetic confinement devices and magnetohydrodynamic equilibrium. A magnetic confinement device uses magnetic fields to confine charged particles that make up a plasma. The magnetic fields are typically generated by current-carrying coils that are located in the vacuum surrounding the plasma. One type of magnetic confinement device is the stellarator, which has its plasma configured to be topologically a solid torus. This paper considers a problem\footnote{The authors would like to thank Per Helander for bringing this problem to our attention.} arising in the study of such devices. Let $\Omega$ be a region of toroidal plasma with a magnetic field $\B$ tangent to the plasma boundary $\partial\Omega$ generated by a collection of external current-carrying coils. How close to the plasma boundary does the nearest coil need to be? Alternatively, this problem can be posed as trying to find how far the magnetic field $\B$ can be externally extended from the plasma boundary, subject to the vacuum field equations (curl and divergence free), before arriving at a singularity. The singularity indicates where a current-carrying coil is expected to be located. This is because in the magnetic fields generated by current-carrying coils, the coils appear as singularities in the magnetic field.

In this paper we focus on a two-dimensional version of this problem. More precisely, let $\Omega\subset\mathbb{R}^2$ be an open bounded and simply connected set with real analytic boundary $\partial\Omega$. We use $\T$ and $\N$ to denote the unit tangent and outward normal vectors on $\partial\Omega$ respectively. Given real-valued functions $f$ and $h$ on $\partial\Omega$, we are interested in finding an open connected $U\subset\mathbb{R}^2$ satisfying $\overline{\Omega} \subset U$ such that there exists a vector field $\B=(B_1,B_2)$ solving the following Cauchy problem 
\begin{subequations}\label{ECP}
\begin{empheq}[left=\empheqlbrace]{alignat=2}
\operatorname{div} \B:= \frac{\partial B_1}{\partial x} + \frac{\partial B_2}{\partial y} &= 0  \quad && \text{in} \ U\setminus \overline{\Omega} \label{DIV}\\
\operatorname{curl} \B:=\frac{\partial B_2}{\partial x} - \frac{\partial B_1}{\partial y} &= 0 && \text{in} \ U\setminus \overline{\Omega} \label{CURL}\\
\B\cdot\T &= f  && \text{on} \ \partial\Omega \label{TANG}\\
\B\cdot\N &= h && \text{on} \ \partial\Omega. \label{NORM}
\end{empheq}
\end{subequations}
We are particularly interested in finding how large we can make the distance between $\partial U$ and $\partial\Omega$. We note that $\partial U$ is external to $\Omega$ since $U$ is taken to satisfy $\overline{\Omega} \subset U$.  

We call $\B$ harmonic if it satisfies equations \eqref{DIV} and \eqref{CURL}. Therefore, this problem is equivalent to finding how far we can harmonically extend $\B$ outwards from $\partial\Omega$. Equations \eqref{DIV} and \eqref{CURL} are the vacuum field equations for a magnetic field, and equations \eqref{TANG} and \eqref{NORM} are the boundary conditions. The physically relevant boundary conditions require that $\B$ is tangent to the plasma boundary $\partial\Omega$ which in our set up means taking $h=0$ and leaving $f$ arbitrary, but we treat the case of a general $h$ for its interesting mathematics.

We now give a couple of useful remarks. Given open connected $U\subset\mathbb{R}^2$ satisfying $\overline{\Omega} \subset U$, the uniqueness of solutions to the Cauchy problem \eqref{ECP} in the class $\B \in {C^1(U\setminus \overline{\Omega};\mathbb{R}^2)} \cap C(U\setminus \Omega;\mathbb{R}^2)$ follows from Holmgren's Uniqueness Theorem \cite[\S2.3]{ReRo}. This theorem will also be useful later on in Section \ref{SEC2} when we have to combine together solutions and make sure they coincide on overlaps. The equations $\operatorname{div} \B = \operatorname{curl} \B = 0$ are the Cauchy--Riemann equations for the complex function $\mathcal{B}:=B_1-iB_2$ with respect to $z=x+iy$. This fact will also come in use later on in section \ref{SEC3}. Note that $\mathcal{B}$ being complex analytic implies that $B_1$ and $B_2$ are harmonic. 

We introduce some required definitions. By identifying the unit circle $\mathbb{T}$ with the interval $[0,2\pi]$, the boundary $\partial\Omega$ being real analytic means there exists a parameterisation $\bm{\gamma}=(\gamma_1,\gamma_2):\mathbb{T}\to \partial\Omega$ that is both real analytic (each component is real analytic) and regular ($\bm{\gamma}'(t)\neq 0$ for all $t\in\mathbb{T}$). Given $t_0\in\mathbb{T}$ such that $\bm{\gamma}(t_0) = \bm{v}_0$, we say that $f$ on $\partial\Omega$ is real analytic at $\bm{v}_0$ if $f(\bm{\gamma}(t))$ is real analytic at $t_0$. It is straightforward to check this definition is independent of the parameterisation chosen. We use $C^\omega(\partial\Omega)$ to denote the set of functions that are real analytic at every point in $\partial\Omega$. 

We now provide a summary of the results in this paper. In Section \ref{SEC2} we prove the boundary data has to satisfy a certain degree of regularity in order for a solution to the Cauchy problem \eqref{ECP} to exist. We prove that, for $f,h \in C^1(\partial\Omega)$, there exists an open connected $U\subset\mathbb{R}^2$ satisfying $\overline{\Omega} \subset U$ and a solution $\B \in C^1(U\setminus \overline{\Omega};\mathbb{R}^2) \cap C(U\setminus \Omega;\mathbb{R}^2)$ to the Cauchy problem \eqref{ECP} if and only if the boundary data satisfies a combined regularity and compatibility condition. The condition is that $f-\mathcal{H}h$ is real analytic on $\partial\Omega$ where
$\mathcal{H}$ is the operator given by
\[ \mathcal{H} h(\bm{v}) := \frac{1}{\pi} \lim_{\varepsilon \to 0}  \int_{\partial\Omega\setminus B_\varepsilon(\bm{v})} h(\bm{w}) \frac{ \bm{t}(\bm{v})\cdot(\bm{v}-\bm{w})}{| \bm{v} - \bm{w} |^2} \, \textrm{d}\bm{w} \]
for $\bm{v}\in \partial\Omega$ and $B_\varepsilon(\bm{v})$ is the ball in $\mathbb{R}^2$ centred at $\bm{v}$ with radius $\varepsilon$. We note that if $\partial\Omega$ is taken to be a straight line, then the operator $\mathcal{H}$ is identical to the standard Hilbert transform. In the case $h=0$, the condition of $f-\mathcal{H}h$ being real analytic simplifies to $f$ being real analytic on $\partial\Omega$; and in this case the proof can be considerably shortened.

Our method of proof involves generalising a similar type of result due to Hadamard on the Cauchy problem for the Laplacian \cite{Had}. The relevance of the Cauchy problem for the Laplacian comes from the fact that on simply connected domains every harmonic vector field can be written as the gradient of a harmonic scalar potential. A survey of results on the Cauchy problem for the Laplacian is presented in \cite{HaVaGo}. Hadamard's result considers the case of a flat boundary. We generalise this result to the case where the boundary data lies on an analytic curve. A detailed version of the proof of Hadamard's result but with the Laplacian replaced by the equation $\partial_{yy} u + y^\alpha\partial_{xx} u =0$, can be found in \cite{PaSa}.  

Section \ref{SEC2} shows us that it is not unreasonable to assume our boundary data is real analytic. Therefore, in Section \ref{SEC3} we assume $f,h\in C^\omega(\partial\Omega)$ and use the Cauchy--Kovalevskaya Theorem to find a lower bound on how far we can solve the Cauchy problem \eqref{ECP}. The lower bound depends on the two functions $\Theta$ and $\Lambda$, which will come to be defined by \eqref{THETA} and \eqref{LAMBDA} respectively. $\Theta$ depends on the parameterisation $\bm{\gamma}$ and boundary data $f$ and $h$, whereas $\Lambda$ only depends on $\bm{\gamma}$. We find that we can solve at least a distance $d^*$ away from $\partial\Omega$ where $d^*$ depends on the the Taylor series coefficients of $\Lambda$ and radius of convergence of the Taylor series of $\Lambda\Theta'$.  We show that the distance $d^*$ is no more than half the minimum radius of curvature, 
\[ d^* \leq \frac{1}{2}\inf_{\mathbb{T}}\left(\frac{1}{\kappa}\right), \]
where $\kappa$ is the curvature of $\bm{\gamma}$. We then conclude with some examples on computing and estimating $d^*$.

\section{Boundary Data Regularity}\label{SEC2}

Given a function $\Psi\colon[-1,1]\to \mathbb{R}$ that has a real analytic extension to an open neighbourhood of $[-1,1]$, we let
\[ \Gamma := \{ (x,\Psi(x)): x\in (-1,1) \} \subset \mathbb{R}^2 \]
be the curve that is the graph of $\Psi$. Since every analytic curve can locally be written as the graph of an analytic function, we initially consider a local version of the Cauchy problem \eqref{ECP} where $\partial\Omega$ is replaced by $\Gamma$.  Let ${\Omega = \{ (x,y)\in (-1,1)\times \mathbb{R} : y<\Psi(x)\}}$, and $\N$ be the unit normal to the curve $\Gamma$ facing away from $\Omega$. 

Since on simply connected domains every harmonic vector field can be written as the gradient of a scalar potential, we can locally find a harmonic scalar potential $u$ satisfying $\B=\nabla u$. In this notation the boundary condition \eqref{TANG} becomes
\[ \nabla u(x,\Psi(x)) \cdot (1,\Psi'(x)) = f(x,\Psi(x))\sqrt{1+\Psi'(x)^2}, \]
which by the Fundamental Theorem of Calculus for Line Integrals can be integrated to obtain $u(x,\Psi(x)) = g(x)$ where $g'(x) = f(x,\Psi(x))\sqrt{1+\Psi'(x)^2}$. Furthermore, the boundary condition \eqref{NORM} becomes $\frac{\partial u}{\partial \N} (x,\Psi(x)) = h(x,\Psi(x))$. To simplify notation we replace $h(x,\Psi(x))$ with $h(x)$. 
 
This shows that the Cauchy problem \eqref{ECP} is in a local sense equivalent to the Cauchy problem for the Laplacian given by \eqref{CP2}. We now introduce the following theorem on the existence of the Cauchy problem for the Laplacian.  
\begin{thm}\label{HAD2}
Let $g,h \in C^1([-1,1])$. There exists $U\subset\mathbb{R}^2$, an open connected neighbourhood of $\Gamma$, and $u\in C^2(U\setminus \overline{\Omega}) \cap C^1(U\setminus \Omega)$ that solves
\begin{subequations}\label{CP2}
\begin{empheq}[left =\empheqlbrace]{alignat=2}
\Delta u(x,y) &=0  & \quad &\text{for} \ (x,y)\in U\setminus \overline{\Omega}\\
u (x,\Psi(x))&= g(x) & &\text{for} \ x\in(-1,1) \label{DIR2}\\
\frac{\partial u}{\partial \N} (x,\Psi(x)) &= h(x) & &\text{for} \ x\in(-1,1) 
\end{empheq}
\end{subequations}
if and only if
\begin{align*}
H(x) := g(x) - \frac{1}{\pi}\int_{-1}^1h(t) \sqrt{1+\Psi'(t)^2} \log\sqrt{(x-t)^2 + (\Psi(x)-\Psi(t))^2} \, \mathrm{d}t
\end{align*}
is real analytic on $(-1,1)$.
\end{thm}

If we were to take $\Gamma$ to be flat ($\Psi=0$), then this theorem recovers a result due to Hadamard \cite{Had}. In the proof of Theorem \ref{HAD2} we will require the following lemma on the analyticity of functions defined by integrals.
\begin{lem}\label{ANAINT}
Suppose that $U\subset \mathbb{C}$ is open and $I\subset \mathbb{R}$ is a compact interval. If the continuous function $\mathcal{A}\colon U\times I \to \mathbb{C}$ is complex analytic in $z\in U$ for each $x\in I$, then $\int_I \mathcal{A}(z,x) \, \mathrm{d}x$ is complex analytic on $U$. 
\end{lem}
\begin{proof}
Let $\Gamma$ be a triangle in $U$. The continuity of $\mathcal{A}$ implies that $\int_I |\mathcal{A}(z,x)| \, \textrm{d}x$ is bounded on $\Gamma$ and so by Fubini's Theorem
\[ \int_\Gamma \int_I \mathcal{A}(z,x) \, \textrm{d}x\,  \textrm{d}z =  \int_I \int_\Gamma \mathcal{A}(z,x)\, \textrm{d}z \, \textrm{d}x = 0.\]
Hence Morera's Theorem \cite[\S5.1]{StSh} implies that $\int_I \mathcal{A}(z,x) \, \textrm{d}x$ is complex analytic on $U$.
\end{proof}
The proof of Theorem \ref{HAD2} follows.
\begin{proof}
We begin by proving that $H$ is necessarily real analytic. Suppose that there exists $u\in C^2(U\setminus \overline{\Omega}) \cap C^1(U\setminus \Omega)$ that solves the Cauchy problem \eqref{CP2}. We aim to show that $H$ is real analytic at a fixed $x_0\in (-1,1)$. To achieve this we construct a region $V_{\delta,\varepsilon}$, shaded in Figure \ref{FIG3} below, and apply a Green's identity \eqref{GI2} over this region to the function $u$ and a constant multiple of the fundamental solution of the Laplacian centred at $(x,\Psi(x))\in\Gamma$. Then we will proceed with analysing the boundary terms of the Green's identity.  

To define the region $V_{\delta,\varepsilon}$ we choose $\delta >0$ small enough such that 
\[\delta<\min\{|x_0+1|,|x_0-1|\}\] 
and the open ball $B_\delta(x_0,\Psi(x_0))$ centred at $(x_0,\Psi(x_0))$ with radius $\delta$ is contained within $U$. If we were to travel anti-clockwise along the circle $\partial B_\delta(x_0,\Psi(x_0))$ starting from the highest point $(x_0,\Psi(x_0)+\delta)$, then eventually we would come into contact with the curve $\Gamma$. Let $x_{0,\delta}^-$ be the $x$-coordinate of the first point of contact. A straightforward compactness argument can be used to prove the existence of such a point. Define $x_{0,\delta}^+$ similarly when the circle is traversed clockwise. We denote by $C_\delta^+$ the arc of $\partial B_\delta(x_0,\Psi(x_0))$ from $(x_{0,\delta}^-,\Psi(x_{0,\delta}^-))$ to $(x_{0,\delta}^+,\Psi(x_{0,\delta}^+))$ that passes through the point $(x_0,\Psi(x_0)+\delta)$.

\begin{figure}[H]
\begin{tikzpicture}[scale=0.95]

\def\a{4.9}
\def\r{3.2}
\def\b{6.3}
\def\rr{0.9}
\def\h{1}

\def\s{4}
\def\A{1.5}

\path [name path=curve] plot [smooth, samples=100, domain=0:10] (\x, {\A*sin(\s*\x r)+\h});
\path [name path=circ1] (\a,{\A*sin(\s*\a r)+\h}) circle [radius=\r];
\path [name path=circ2] (\b,{\A*sin(\s*\b r)+\h}) circle [radius=\rr];

\begin{scope}
\clip (0,6) -- plot [smooth, samples=100, domain=0:10] (\x, {\A*sin(\s*\x r)+\h}) -- (10,6) -- cycle;
\draw [thick, fill=gray!20] (\a,{\A*sin(\s*\a r)+\h}) circle [radius=\r];
\end{scope}

\draw [thick, dotted, name intersections={of=curve and circ1, name=large}] 
(large-1) -- (large-1 |- 0,-1.8)
(large-6) -- (large-6 |- 0,-1.8);

\begin{scope}
\clip (0,5) -- plot [smooth, samples=100, domain=0:10] (\x, {\A*sin(\s*\x r)+\h}) -- (10,5) -- cycle;
\fill [gray!20] (large-1) -- (large-1 |- 0,-2) -- (large-6 |- 0,-2) -- (large-6) -- cycle;
\end{scope}

\draw [thick, dotted, name intersections={of=curve and circ2, name=small}] 
(small-1) -- (small-1 |- 0,-1.8) 
(small-4) -- (small-4 |- 0,-1.8); 

\begin{scope}
\clip (0,5) -- plot [smooth, samples=100, domain=0:10] (\x, {\A*sin(\s*\x r)+\h}) -- (10,5) -- cycle;
\clip  (small-1 |- 0,-2) --  (small-1 |- 0,5) -- (small-4 |- 0,5) -- (small-4 |- 0,-2) -- cycle;    
\draw [thick, fill=white] (\b,{\A*sin(\s*\b r)+\h}) circle [radius=\rr];
\fill [white]  (small-1 |- 0,-2) -- (small-1) -- (small-4) -- (small-4 |- 0,-2) -- cycle;
\end{scope}

\draw [thick] plot [smooth, samples=100, domain=0.5:9.5] (\x, {\A*sin(\s*\x r)+\h}) node [right] {$\Gamma$};

\draw [thick] (0,-2) -- (10,-2);
\draw [thick, dotted] (\a,{\A*sin(\s*\a r)+\h}) -- (\a,-1.8) ;
\draw (\a,-1.8) -- (\a,-2.2) node [below] {$x_0$};
\draw [thick, dotted] (\b,{\A*sin(\s*\b r)+\h}) -- (\b,-1.8);
\draw (\b,-1.8) -- (\b,-2.2) node [below] {$x$};

\draw (large-1 |- 0,-1.8) --  (large-1 |- 0,-2.2) node [below=-0.1cm] {$x_{0,\delta}^-$}; 
\draw (large-6 |- 0,-1.8) --  (large-6 |- 0,-2.2) node [below=-0.1cm] {$x_{0,\delta}^+$}; 

\draw (small-1 |- 0,-1.8) --  (small-1 |- 0,-2.2) node [below=-0.1cm] {$x_\varepsilon^-$}; 
\draw (small-4 |- 0,-1.8) --  (small-4 |- 0,-2.2) node [below right=-0.1cm] {$x_\varepsilon^+$};

\node [above] at ($(\a,{\A*sin(\s*\a r)+\h})+(90:\r)$) {$C_\delta^+$};
\node [above] at ($(\b,{\A*sin(\s*\b r)+\h})+(110:\rr)$) {$C_\varepsilon^-$};
\node [left] at (3.2,{\A*sin(\s*3.2 r)+\h}) {$L_{\delta,\varepsilon}^-$};
\node [right] at (7,{\A*sin(\s*7 r)+\h}) {$L_{\delta,\varepsilon}^+$};
\node at (4.5,2.5+\h) {$V_{\delta,\varepsilon}$};

\draw (0.5,-1.8) -- (0.5,-2.2) node [below] {$-1$};
\draw (9.5,-1.8) -- (9.5,-2.2) node [below] {$1$};

\draw [thick, dashed] plot [smooth] coordinates {(0.5,{\A*sin(\s*0.5 r)+\h}) (1,6+\h) (3,5+\h) (6,6.5+\h) (9,5+\h) (9.5,{\A*sin(\s*9.5 r)+\h})};

\node at (8,5+\h) {$U\setminus \overline{\Omega}$};
\node at (3,-2.5+\h) {$\Omega$};

\draw [fill] (\a,{\A*sin(\s*\a r)+\h}) circle [radius=2pt];
\draw [fill] (\b,{\A*sin(\s*\b r)+\h}) circle [radius=2pt];

\draw [dashed] (\a,{\A*sin(\s*\a r)+\h}) circle [radius=\r];
\draw [dashed] (\b,{\A*sin(\s*\b r)+\h}) circle [radius=\rr];

\end{tikzpicture}
\centering
\caption{Diagram of the region $V_{\delta,\varepsilon}$.}
\label{FIG3}
\end{figure}
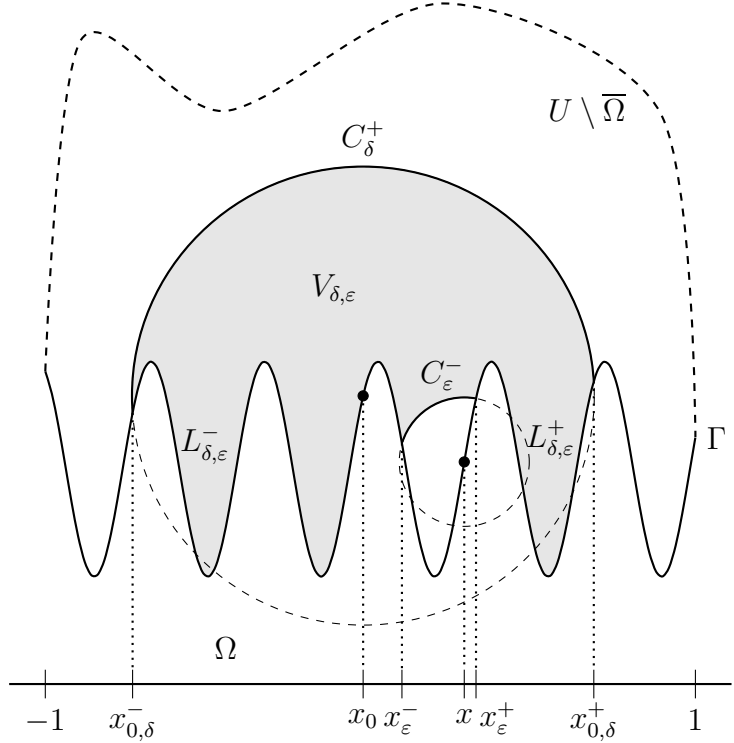

Let $x\in (x_{0,\delta}^-,x_{0,\delta}^+)$ and define $\Phi_x \colon \mathbb{R}^2\setminus \{(x,\Psi(x))\} \to \mathbb{R}$ to be the following constant multiple of the fundamental solution of the Laplacian, 
\[ \Phi_x(\xi,\eta):= \log |(x,\Psi(x))-(\xi,\eta)| = \log \sqrt{(x-\xi)^2+(\Psi(x)-\eta)^2}. \]
To avoid the singularity of $\Phi_x$ we need to cut out a region around the point $(x,\Psi(x))$. Given $0<\varepsilon <\min\{|x-x_{0,\delta}^-|,|x-x_{0,\delta}^+|\}$ we find $x_\varepsilon^-$, $x_\varepsilon^+$, $C_\varepsilon^-$ from $x$ in the same way as we constructed $x_{0,\delta}^-$, $x_{0,\delta}^+$, $C_\delta^+$ from $x_0$. We now define the curves
\[ L_{\delta,\varepsilon}^- := \{(t,\Psi(t)) : t\in [x_{0,\delta}^-, x_\varepsilon^-] \}  \]
and
\[ L_{\delta,\varepsilon}^+ := \{(t,\Psi(t)) : t\in [x_\varepsilon^+, x_{0,\delta}^+] \}, \]
which are the segments of $\Gamma$ connecting $C_\varepsilon^-$ and $C_\delta^+$.
Then we define the open set $V_{\delta,\varepsilon}$, shaded in Figure \ref{FIG3}, as the region contained in $U\setminus\overline{\Omega}$ bounded by the curves $L_{\delta,\varepsilon}^-$, $C_\varepsilon^-$, $L_{\delta,\varepsilon}^+$, and $C_\delta^+$. 

Due to the way we have constructed $V_{\delta,\varepsilon}$, it has a piecewise $C^2$ boundary $\partial V_{\delta,\varepsilon}$. To apply a Green's identity over $V_{\delta,\varepsilon}$, we will need to translate the set slightly upwards since we are not assuming that $u$ is $C^2$ up to $\Gamma$. Let $\rho>0$ be small enough such that $V_{\delta,\varepsilon,\rho} := V_{\delta,\varepsilon} + (0,\rho)$ is contained within $U\setminus\overline{\Omega}$, which means that $u,\Phi_x \in C^2(\overline{V_{\delta,\varepsilon,\rho}})$. We can then apply a Green's identity to obtain
\begin{equation}\label{GI2}
\int_{\partial V_{\delta,\varepsilon,\rho}} \left( u \frac{\partial\Phi_x}{\partial\nu} - \Phi_x \frac{\partial u}{\partial\nu} \right) \, \textrm{d}s = \int_{V_{\delta,\varepsilon,\rho}} (u\Delta \Phi_x - \Phi_x\Delta u) \, \textrm{d}A,
\end{equation}
where $\nu$ is the unit outward normal to $V_{\delta,\varepsilon,\rho}$. Since both $u$ and $\Phi_x$ are harmonic in $V_{\delta,\varepsilon,\rho}$, the right-hand side of equation \eqref{GI2} vanishes. Hence as the integrand on the left-hand side is continuous up to $\Gamma$, we can take the limit as $\rho$ tends to zero to obtain
\begin{equation}\label{GI3}
 \int_{\partial V_{\delta,\varepsilon}} \left( u \frac{\partial\Phi_x}{\partial\nu} - \Phi_x \frac{\partial u}{\partial\nu} \right) \, \textrm{d}s = 0. 
\end{equation}

We now split the above integral into four, using the different pieces of boundary $\partial V_{\delta,\varepsilon}$, and separately evaluate their limiting values as $\varepsilon \to 0$. When analysing the integrals over $C_\varepsilon^-$, we will need to know the limiting value of $|C_\varepsilon^-|/\varepsilon$ as $\varepsilon\to 0$ where $|C_\varepsilon^-|$ is the length of $C_\varepsilon^-$. Notice that the manner in which we constructed the arc $C_\varepsilon^-$, implies that
\begin{align*}
\frac{|C_\varepsilon^-|}{\varepsilon} &= \pi  - \tan^{-1} \left( \frac{\Psi(x_\varepsilon^+) -\Psi(x)}{x_\varepsilon^+ - x} \right) + \tan^{-1} \left( \frac{\Psi(x_\varepsilon^-) -\Psi(x)}{x_\varepsilon^- - x} \right).
\end{align*}
It follows $|C_\varepsilon^-|/\varepsilon \to \pi$ as $\varepsilon\to 0$ because $x_\varepsilon^-\to x$ and $x_\varepsilon^+\to x$ as $\varepsilon\to 0$. Therefore, since $\frac{\partial\Phi_x}{\partial\nu}=-1/\varepsilon$ on $C_\varepsilon^-$, we have
\[  \int_{C_\varepsilon^-} u \frac{\partial\Phi_x}{\partial\nu} \, \textrm{d}s  = -\frac{1}{\varepsilon}\int_{C_\varepsilon^-}  u \, \textrm{d}s = -\frac{|C_\varepsilon^-|}{\varepsilon}\frac{1}{|C_\varepsilon^-|}\int_{C_\varepsilon^-} u \, \textrm{d}s \to -\pi g(x) \]
as $\varepsilon \to 0$ using the boundary condition \eqref{DIR2}. Furthermore, if we let $K$ be a compact neighbourhood of $(x,\Psi(x))$ in $U\setminus \Omega$, then once $\varepsilon$ is small enough we have
\begin{align*}
\left|  \int _{C_\varepsilon^-} \Phi_x \frac{\partial u}{\partial\nu} \, \textrm{d}s \right| &\leq \Big( \sup_{K} |\nabla u| \Big) \int_{C_\varepsilon^-} |\Phi_x| \, \textrm{d}s \\
&=  \Big(\sup_{K} |\nabla u | \Big) \frac{|C_\varepsilon^-|}{\varepsilon} \varepsilon |\log(\varepsilon)| \\
&\to 0
\end{align*}
as $\varepsilon\to 0$. As a result, this term will not contribute to the limiting value of equation \eqref{GI3}.   

Next, we evaluate the integrals over $L_{\delta,\varepsilon}^-$ and $L_{\delta,\varepsilon}^+$. Firstly,
\begin{align*}
\int_{L_{\delta,\varepsilon}^-} \Phi_x \frac{\partial u}{\partial\nu}  \, \textrm{d}s &= \int_{x_{0,\delta}^-}^{x_\varepsilon^-} \Phi_x(t,\Psi(t)) \frac{\partial u}{\partial\nu} (t,\Psi(t)) \sqrt{1+\Psi'(t)^2} \, \textrm{d}t \\
&= -\int_{x_{0,\delta}^-}^{x_\varepsilon^-} h(t) \sqrt{1+\Psi'(t)^2} \log \sqrt{(x-t)^2+(\Psi(x)-\Psi(t))^2}  \, \textrm{d}t.
\end{align*}
We now show that when taking the limit of this expression as $\varepsilon\to 0$, the upper limit of the integral changes to $x$. Using the existence of $\phi\in C([-1,1]^2)$ satisfying 
\[\Psi(x)-\Psi(t) = \phi(x,t)(x-t), \]
we have
\begin{align*}
&\left| \int_{x_\varepsilon^-}^{x} h(t) \sqrt{1+\Psi'(t)^2} \log \sqrt{(x-t)^2+(\Psi(x)-\Psi(t))^2}  \, \textrm{d}t \right| \\
&\qquad \leq \int_{x_\varepsilon^-}^{x} |h(t)| \sqrt{1+\Psi'(t)^2} \left|\log|x-t| + \log\sqrt{1+\phi(x,t)^2}\right|  \, \textrm{d}t \\
&\qquad \leq \sup_{[-1,1]} \Big(|h| \sqrt{1+\Psi'^2}\Big) \int_{x-\varepsilon}^{x} \big|\log|x-t|\big| + \log\sqrt{1+\phi(x,t)^2}  \, \textrm{d}t \\
&\qquad \to 0
\end{align*}
as $\varepsilon\to 0$ where we have used the fact that $\log|x-\cdot|\in L^1([-1,1])$. Therefore,
\[ \int_{L_{\delta,\varepsilon}^-} \Phi_x \frac{\partial u}{\partial\nu}  \, \textrm{d}s \to -\int_{x_{0,\delta}^-}^{x} h(t) \sqrt{1+\Psi'(t)^2} \log \sqrt{(x-t)^2+(\Psi(x)-\Psi(t))^2}  \, \textrm{d}t \]
as $\varepsilon\to 0$. Similarly we have
\[ \int_{L_{\delta,\varepsilon}^+} \Phi_x \frac{\partial u}{\partial\nu}  \, \textrm{d}s \to -\int_x^{x_{0,\delta}^+} h(t) \sqrt{1+\Psi'(t)^2} \log \sqrt{(x-t)^2+(\Psi(x)-\Psi(t))^2}  \, \textrm{d}t \]
as $\varepsilon\to 0$. 

We also have
\[  \int_{L_{\delta,\varepsilon}^-} u \frac{\partial\Phi_x}{\partial\nu} \, \textrm{d}s = \int_{x_{0,\delta}^-}^{x_\varepsilon^-} u(t,\Psi(t)) \left[ \frac{\Psi(x)-\Psi(t)-\Psi'(t)(x-t)}{(x-t)^2+(\Psi(x)-\Psi(t))^2} \right] \, \textrm{d}t.  \]
Notice that by using the fact that $\Psi$ is twice continuously differentiable, the integrand has a continuous extension to $t=x$. Consequently, since $u=g$ on $\Gamma$, we have
\[   \int_{L_{\delta,\varepsilon}^-} u \frac{\partial\Phi_x}{\partial\nu} \, \textrm{d}s \to \int_{x_{0,\delta}^-}^{x} g(t) \left[ \frac{\Psi(x)-\Psi(t)-\Psi'(t)(x-t)}{(x-t)^2+(\Psi(x)-\Psi(t))^2} \right] \, \textrm{d}t \]
as $\varepsilon \to 0$, and similarly 
\[ \int_{L_{\delta,\varepsilon}^+} u \frac{\partial\Phi_x}{\partial\nu} \, \textrm{d}s \to \int_x^{x_{0,\delta}^+} g(t) \left[ \frac{\Psi(x)-\Psi(t)-\Psi'(t)(x-t)}{(x-t)^2+(\Psi(x)-\Psi(t))^2} \right] \, \textrm{d}t \]
as $\varepsilon\to0$. Substituting everything into equation \eqref{GI3} and then taking the limit as $\varepsilon\to 0$ results in 
\begin{equation}\label{EQ1}
\begin{split}
&g(x) - \frac{1}{\pi}\int_{x_{0,\delta}^-}^{x_{0,\delta}^+} h(t) \sqrt{1+\Psi'(t)^2} \log \sqrt{(x-t)^2+(\Psi(x)-\Psi(t))^2}  \, \textrm{d}t \\
&= \frac{1}{\pi}\int_{C_\delta^+} \left( u \frac{\partial\Phi_x}{\partial\nu} - \Phi_x \frac{\partial u}{\partial\nu} \right) \, \textrm{d}s + \frac{1}{\pi}\int_{x_{0,\delta}^-}^{x_{0,\delta}^+} g(t) \left[ \frac{\Psi(x)-\Psi(t)-\Psi'(t)(x-t)}{(x-t)^2+(\Psi(x)-\Psi(t))^2} \right] \, \textrm{d}t.
\end{split}
\end{equation}
We would like to show that the right-hand side of this equation is real analytic at $x=x_0$. 

We start by noting that Taylor's Theorem with integral remainder yields
\[  \Psi(x) = \Psi(t) + \int_0^1 \Psi'((1-\tau)x +\tau t) \, \textrm{d}\tau (x-t) \]
and
\[ \Psi(x) = \Psi(t) + \Psi'(t)(x-t) + \int_0^1 \Psi''((1-\tau)x +\tau t)\tau \, \textrm{d}\tau (x-t)^2. \] 
Therefore, by setting
\begin{align*}
\mathcal{A}(x,t) &:= \int_0^1 \Psi'((1-\tau)x +\tau t) \, \textrm{d}\tau, \\
\mathcal{B}(x,t) &:= \int_0^1 \Psi''((1-\tau)x +\tau t)\tau \, \textrm{d}\tau,
\end{align*} 
we can write
\[ \frac{\Psi(x)-\Psi(t)-\Psi'(t)(x-t)}{(x-t)^2+(\Psi(x)-\Psi(t))^2} = \frac{\mathcal{B}(x,t)}{1+\mathcal{A}(x,t)^2}. \]

Since $\Psi$ is real analytic, there exists an open neighbourhood of $[-1,1]$ in the complex plane $\mathbb{C}$ where $\Psi$ is complex analytic. We can then choose $a>0$ small enough such that the closure of 
\[ R:= {\{ z\in\mathbb{C} : \RE (z)\in(-1,1), \IM (z)\in (-a,a) \}} \] 
lies within the region where $\Psi$ is complex analytic. Thus the expressions $\Psi'((1-\tau)z +\tau t)$ and $ {\Psi''((1-\tau)z +\tau t)\tau}$ are complex analytic in $z\in R$ for all $t\in [-1,1]$, $\tau\in [0,1]$. Hence, we can use Lemma \ref{ANAINT} to guarantee that both $\mathcal{A}(z,t)$ and $\mathcal{B}(z,t)$ are complex analytic on $R$ for all $t\in [-1,1]$. We now choose $a$ small enough such that $|\IM(\Psi'(z))|<1/2$ for all $z\in R$, to acquire the bound
\begin{align*}
|1+\mathcal{A}(z,t)^2| &\geq |\RE(1+\mathcal{A}(z,t)^2)| \\
&= 1 + (\RE\mathcal{A}(z,t))^2 - (\IM\mathcal{A}(z,t))^2 \\
&\geq \frac{3}{4}  
\end{align*} 
for $z\in R$ and $t\in[-1,1]$. Then $ \frac{\mathcal{B}(z,t)}{1+\mathcal{A}(z,t)^2}$ is complex analytic on $R$ for each $t\in [-1,1]$. Once again we can apply Lemma \ref{ANAINT}, this time to justify the complex analyticity of
\[ \int_{x_{0,\delta}^-}^{x_{0,\delta}^+} g(t) \frac{\mathcal{B}(z,t)}{1+\mathcal{A}(z,t)^2} \, \mathrm{d}t \]
on $R$. It follows that
\begin{equation}\label{ANAINT2}
\int_{x_{0,\delta}^-}^{x_{0,\delta}^+} g(t) \left[ \frac{\Psi(x)-\Psi(t)-\Psi'(t)(x-t)}{(x-t)^2+(\Psi(x)-\Psi(t))^2} \right] \, \mathrm{d}t 
\end{equation}
is real analytic at $x=x_0$.

We still need to show that
\[ \int_{C_\delta^+} \left( u \frac{\partial\Phi_x}{\partial\nu} - \Phi_x \frac{\partial u}{\partial\nu} \right) \, \mathrm{d}s \]
is real analytic at $x=x_0$. For some $\theta_1<\theta_2$ depending on $\delta$, we have
\[  \int_{C_\delta^+} \Phi_x \frac{\partial u}{\partial\nu} \, \mathrm{d}s = \int_{\theta_1}^{\theta_2} \frac{\delta}{2}\log (\mathcal{E}(x,t)) \frac{\partial u}{\partial\nu} (x_0+\delta\cos t, \Psi(x_0)+\delta\sin t ) \, \mathrm{d}t\]
where $\mathrm{d}s = \delta \, \mathrm{d}t$ and
\[ \mathcal{E}(x,t) = (x-x_0 -\delta\cos t)^2 + (\Psi(x)-\Psi(x_0)-\delta\sin t)^2. \]
Let $0<r<\frac{\delta}{3\sqrt{2}}$ be small enough such that the complex disc 
\[ D_r(x_0) = \{ z\in\mathbb{C}:|z-x_0|<r\} \] 
lies within the region of complex analyticity of $\Psi$ and $|\Psi(z)-\Psi(x_0)|< \frac{\delta}{3\sqrt{2}}$ for all $z\in D_r(x_0)$. Then, because for all $t\in[\theta_1,\theta_2]$ either $|\cos t|>\frac{1}{\sqrt{2}}$ or $|\sin t|>\frac{1}{\sqrt{2}}$, it follows that
\begin{align*}
\RE \mathcal{E}(z,t) &= (\RE(z-x_0 -\delta\cos t))^2 - (\IM(z-x_0 -\delta\cos t))^2 \\
& \qquad + (\RE(\Psi(z)-\Psi(x_0)-\delta\sin t))^2 - (\IM(\Psi(z)-\Psi(x_0)-\delta\sin t))^2 \\
&\geq \left( \frac{\delta}{\sqrt{2}}-  \frac{\delta}{3\sqrt{2}} \right)^2 - (\IM(z-x_0))^2  - (\IM(\Psi(z)-\Psi(x_0)))^2 \\
&\geq \frac{2\delta^2}{9} - 2\left( \frac{\delta}{3\sqrt{2}}\right)^2 \\
&= \frac{\delta^2}{9}.
\end{align*}
Now if we take $\log$ to be the principle value complex logarithm defined away from the negative real axis, then $\log(\mathcal{E}(z,t))$ is complex analytic on $D_r(x_0)$ for all $t\in [\theta_1,\theta_2]$. Hence by Lemma \ref{ANAINT}, $ \int_{C_\delta^+} \Phi_z \frac{\partial u}{\partial\nu} \, \mathrm{d}s $ is complex analytic for $z\in D_r(x_0)$ and so real analytic at the point of interest $z=x_0$. Similar arguments can be employed to show that
\begin{gather*}
\int_{C_\delta^+} u \frac{\partial\Phi_x}{\partial\nu} \, \mathrm{d}s, \\
\int_{-1}^{x_{0,\delta}^-} h(t) \sqrt{1+\Psi'(t)^2} \log \sqrt{(x-t)^2+(\Psi(x)-\Psi(t))^2}  \, \mathrm{d}t, \\
\int_{x_{0,\delta}^+}^1 h(t) \sqrt{1+\Psi'(t)^2} \log \sqrt{(x-t)^2+(\Psi(x)-\Psi(t))^2}  \, \mathrm{d}t 
\end{gather*}
are all real analytic at $x=x_0$. The above together with equation \eqref{EQ1} conclude our proof that $H$ is real analytic at $x_0$ and thus the entirety of $(-1,1)$, since $x_0$ was arbitrary. \bigbreak

We now prove the sufficiency of $H$ being real analytic on $(-1,1)$. Let $U_1:= (-1,1)\times \mathbb{R}$ and consider $G\colon U_1 \setminus \Omega \to \mathbb{R}$ defined by
\[ G(x,y) :=  \frac{1}{\pi}\int_{-1}^1 h(t) \sqrt{1+\Psi'(t)^2} \log\sqrt{(x-t)^2 + (y-\Psi(t))^2} \, \mathrm{d}t. \]
Notice that the integrand
\[ h(t) \sqrt{1+\Psi'(t)^2}\log\sqrt{(x-t)^2 + (y-\Psi(t))^2} \] 
and all its partial derivatives with respect to $x$ and $y$ are continuous in $(x,y,t) \in (U_1\setminus \overline{\Omega}) \times [-1,1]$. We can thus interchange integral and partial derivative to justify $G$ belonging to $C^\infty(U_1\setminus \overline{\Omega})$. Furthermore, $\Delta G =0$ which suggests, as we will come to discover, that $G$ can be used to construct a solution $u$ to the Cauchy problem \eqref{CP2}. We remark that if $h=0$, then this step can be skipped as $G=0$. Therefore, for this part of the proof we can assume $h\neq0$, and we will find useful to do so.  

We will now show that $G$ is continuous up to $\Gamma$. Let $x_0\in (-1,1)$ and $(x,y) \in U_1\setminus \overline{\Omega}$. We use $f\lesssim g$ to denote the existence of a constant $C$ such that $f\leq Cg$. Observe
\begin{align}
&|G(x,y) - G(x_0,\Psi(x_0)) | \nonumber\\
&\qquad = \left| \frac{1}{\pi}\int_{-1}^1 h(t) \sqrt{1+\Psi'(t)^2} \log\sqrt{(x-t)^2 + (y-\Psi(t))^2} \, \mathrm{d}t \right. \nonumber\\
&\qquad \qquad \left. -\frac{1}{\pi}\int_{-1}^1 h(t) \sqrt{1+\Psi'(t)^2} \log\sqrt{(x_0-t)^2 + (\Psi(x_0)-\Psi(t))^2} \, \mathrm{d}t \right| \nonumber\\
&\qquad \lesssim \int_{-1}^1 \left| \log ((x-t)^2 + (y-\Psi(t))^2) - \log((x_0-t)^2 + (\Psi(x_0)-\Psi(t))^2)\right| \, \mathrm{d}t \nonumber\\
&\qquad \leq \int_{-1}^1 \left| \log( (x-t)^2) - \log((x_0-t)^2) \right| \, \mathrm{d}t \label{INT1} \\
&\qquad \qquad + \int_{-1}^1 \left| \log \left(1 + \frac{(y-\Psi(t))^2}{(x-t)^2}\right) - \log\left(1 + \frac{(\Psi(x_0)-\Psi(t))^2}{(x_0-t)^2}\right)\right| \, \mathrm{d}t. \label{INT2}
\end{align}
The integral in \eqref{INT1} converges to zero as $(x,y)\to (x_0,\Psi(x_0))$ because the $L^1$ norm is continuous with respect to translations. It remains to prove the integral in \eqref{INT2} also tends to zero. We carry out the substitution $s=x-t$, and by supposing $(x,y)$ is sufficiently close to $(x_0,\Psi(x_0))$, there exists constants $C_1,C_2>1$ such that
\[ 1+ \frac{(y-\Psi(x-s))^2}{s^2} \leq \frac{C_1}{s^2} \qquad \text{and} \qquad 1 + \frac{(\Psi(x_0)-\Psi(x-s))^2}{(x_0-(x-s))^2} \leq C_2 \]
for almost every $s\in [x-1,x+1]$. We then dominate the integrand as follows
\begin{align*}
&\chi_{[x-1,x+1]} (s) \left| \log\left( 1+ \frac{(y-\Psi(x-s))^2}{s^2}\right) - \log\left(1 + \frac{(\Psi(x_0)-\Psi(x-s))^2}{(x_0-(x-s))^2}\right)\right| \\
&\qquad \leq |\log(C_1/s^2)| + \log C_2\\
&\qquad \leq \log C_1 + \log C_2 +2\big|\log|s| \big|,
\end{align*}
which lies within $L^1([-2,2])$. By applying the Dominated Converge Theorem we finish our proof that $|G(x,y) - G(x_0,\Psi(x_0)) | \to 0$ as $(x,y) \to (x_0,\Psi(x_0))$. Notice that we have shown $G(x,\Psi(x)) = g(x) - H(x)$ on $(-1,1)$. 

We will also show that the first order derivatives of $G$ can be continuously extended to $\Gamma$. We start by observing that
\begin{align}
\nabla G (x,y) \cdot \N (x) &= \left( \frac{\partial G}{\partial x} (x,y), \frac{\partial G}{\partial y} (x,y)\right) \cdot \frac{1}{\sqrt{1+\Psi'(x)^2}} (-\Psi'(x),1) \nonumber \\
&=  \frac{1}{\sqrt{1+\Psi'(x)^2}} \left( \frac{\partial G}{\partial y} (x,y) -\Psi'(x) \frac{\partial G}{\partial x} (x,y) \right) \label{GDOTN}
\end{align} 
for $(x,y)\in U_1\setminus \overline{\Omega}$. We aim to show the right-hand side of equation \eqref{GDOTN} is continuous up to $\Gamma$. We can write
\begin{align*}
\frac{\partial G}{\partial y} (x,y) -\Psi'(x) \frac{\partial G}{\partial x} (x,y) = \frac{1}{\pi} \int_{-1}^1  h(t) \sqrt{1+\Psi'(t)^2} \left[\frac{y-\Psi(t)-\Psi'(x)(x-t)}{(x-t)^2+(y-\Psi(t))^2}\right] \, \mathrm{d}t. 
\end{align*} 
Note
\begin{align*}
&\frac{y-\Psi(t)-\Psi'(x)(x-t)}{(x-t)^2+(y-\Psi(t))^2} \\
&\qquad = \frac{y-\Psi(x)}{(x-t)^2+(y-\Psi(t))^2} + \frac{\Psi(x) - \Psi(t) - \Psi'(x)(x-t)}{(x-t)^2+(y-\Psi(t))^2} \\
&\qquad = \frac{y-\Psi(x)}{(x-t)^2+[y-\Psi(x)+\phi(x,t)(x-t)]^2} + \frac{\Psi(x) - \Psi(t) - \Psi'(x)(x-t)}{(x-t)^2+(y-\Psi(t))^2}, 
\end{align*}
where
\[ \phi(x,t) = \Psi'(x) - \int_0^1 \Psi''((1-\tau)t +\tau x) \tau\, \mathrm{d}\tau (x-t). \]
We therefore define
\[ \mathcal{I}_1(x,y,t) := \frac{y-\Psi(x)}{(x-t)^2+[y-\Psi(x) + \phi(x,t)(x-t)]^2} \]
and
\[ \mathcal{I}_2(x,y,t) := \frac{\Psi(x) - \Psi(t) - \Psi'(x)(x-t)}{(x-t)^2+(y-\Psi(t))^2}.  \]

We first investigate the limit of the integral $\frac{1}{\pi}\int_{-1}^1 h(t) \sqrt{1+\Psi'(t)^2} \, \mathcal{I}_1(x,y,t) \, \mathrm{d}t$ as $(x,y) \to (x_0,\Psi(x_0))$. The difference between $\mathcal{I}_1$ and
\[ \widetilde{\mathcal{I}}_1(x,y,t) :=  \frac{y-\Psi(x)}{(x-t)^2+[y-\Psi(x) + \phi(x,x)(x-t)]^2} \]
is
\begin{align*}
&\big| \mathcal{I}_1(x,y,t) - \widetilde{\mathcal{I}}_1(x,y,t) \big| \\
&= \frac{\big|(y-\Psi(x))(\phi(x,x)^2-\phi(x,t)^2)(x-t)^2 + 2(\phi(x,x)-\phi(x,t))(y-\Psi(x))^2(x-t)\big|}{\big|(x-t)^2+[y-\Psi(x) + \phi(x,t)(x-t)]^2\big|\big|(x-t)^2+[y-\Psi(x) + \phi(x,x)(x-t)]^2\big|} \\
&\leq \frac{\big|\int_0^1 \Psi''((1-\tau)t +\tau x) \tau\, \mathrm{d}\tau (\phi(x,x) +\phi(x,t))(y-\Psi(x))(x-t)^3\big|}{\big|(x-t)^2+[y-\Psi(x) + \phi(x,t)(x-t)]^2\big|\big|(x-t)^2+[y-\Psi(x) + \phi(x,x)(x-t)]^2\big|} \\
& \qquad +\frac{2\big|\int_0^1 \Psi''((1-\tau)t +\tau x) \tau\, \mathrm{d}\tau(y-\Psi(x))^2(x-t)^2\big|}{\big|(x-t)^2+[y-\Psi(x) + \phi(x,t)(x-t)]^2\big|\big|(x-t)^2+[y-\Psi(x) + \phi(x,x)(x-t)]^2\big|} \\
& \lesssim  \frac{ |y-\Psi(x)||x-t|}{(x-t)^2+[y-\Psi(x) + \phi(x,x)(x-t)]^2} \\
& \qquad +\frac{ |y-\Psi(x)|^2|x-t|^2}{\big|(x-t)^2+[y-\Psi(x) + \phi(x,t)(x-t)]^2\big|\big|(x-t)^2+[y-\Psi(x) + \phi(x,x)(x-t)]^2\big|}. 
\end{align*}
We can bound this above using
\begin{align*}
&\frac{ |y-\Psi(x)||x-t|}{(x-t)^2+[y-\Psi(x) + \phi(x,x)(x-t)]^2} \\
&\qquad \leq \frac{|y-\Psi(x)+\phi(x,x)(x-t)||x-t|}{(x-t)^2+[y-\Psi(x) + \phi(x,x)(x-t)]^2} + \frac{|\phi(x,x)||x-t|^2}{(x-t)^2+[y-\Psi(x) + \phi(x,x)(x-t)]^2} \\
&\qquad \leq 1 + |\phi(x,x)| \\
&\qquad \leq 1 +\sup_{x\in[-1,1]} |\phi(x,x)|
\end{align*}
and the similar estimate
\begin{align*}
\frac{ |y-\Psi(x)||x-t|}{(x-t)^2+[y-\Psi(x) + \phi(x,t)(x-t)]^2} \leq 1 +\sup_{(x,t)\in[-1,1]^2} |\phi(x,t)|.
\end{align*}
Overall, $\big| \mathcal{I}_1(x,y,t) - \widetilde{\mathcal{I}}_1(x,y,t) \big|$ is bounded on $ (U_1\setminus \overline{\Omega}) \times [-1,1]$.  Seeing that $\mathcal{I}_1(x,y,t) \to 0$ and $\widetilde{\mathcal{I}}_1(x,y,t) \to 0$ as $(x,y) \to (x_0,\Psi(x_0))$ for almost every $t\in[-1,1]$, justifies being able to apply the Dominated Convergence Theorem to obtain
\begin{equation}\label{LIMIT1}
\int_{-1}^1 h(t) \sqrt{1+\Psi'(t)^2} \left( \mathcal{I}_1(x,y,t) - \widetilde{\mathcal{I}}_1(x,y,t) \right) \, \mathrm{d}t \to 0 
\end{equation}
as $(x,y) \to (x_0,\Psi(x_0))$.

Hence we can now focus our attention on the integral of $\widetilde{\mathcal{I}}_1$. After substituting $\phi(x,x) = \Psi'(x)$ and some rearrangement we find
\[ \widetilde{\mathcal{I}}_1(x,y,t) = \frac{1}{y-\Psi(x)} \cdot \frac{1+\Psi'(x)^2}{\left((1+\Psi'(x)^2)\frac{x-t}{y-\Psi(x)} +\Psi'(x)\right)^2 +1}, \] 
which when integrated over $\mathbb{R}$ yields
\begin{align*}
\int_{-\infty}^\infty \widetilde{\mathcal{I}}_1(x,y,t) \, \mathrm{d}t &= \int_{-\infty}^\infty \frac{1+\Psi'(x)^2}{\left[(1+\Psi'(x)^2)t +\Psi'(x)\right]^2 +1} \, \mathrm{d}t \\
&=  \tan^{-1}\big((1+\Psi'(x)^2)t+\Psi'(x)\big) \Big|_{t=-\infty}^{t=\infty} \\
&= \pi.
\end{align*}
We can use this fact to show $\widetilde{\mathcal{I}}_1$ behaves like an approximation to the identity as $(x,y)\to (x_0,\Psi(x_0))$. Given $\varepsilon>0$, there exists $0<\eta<\min\{|x_0-1|,|x_0+1|\}$ small enough such that
\[ \big|h(t) \sqrt{1+\Psi'(t)^2} - h(x_0) \sqrt{1+\Psi'(x_0)^2} \big| < \frac{\varepsilon}{2} \]
whenever $|t-x_0|<\eta$. Now suppose $(x,y) \in U_1\setminus \overline{\Omega}$ satisfies 
\[ |(x,y) - (x_0,\Psi(x_0))|< \eta/2.\] 
Then
\begin{align*}
\int_{\mathbb{R}\setminus B_\eta(x_0)} \widetilde{\mathcal{I}}_1(x,y,t) \, \mathrm{d}t &\leq \int_{\mathbb{R}\setminus B_{\frac{\eta}{2}}(x)} \widetilde{\mathcal{I}}_1(x,y,t) \, \mathrm{d}t \\
&= \int_{\mathbb{R}\setminus B_{\frac{\eta}{2(y-\Psi(x))}}(0)}  \frac{1+\Psi'(x)^2}{\left[(1+\Psi'(x)^2)t +\Psi'(x)\right]^2 +1} \, \mathrm{d}t \\
&= \pi -  \tan^{-1}\left((1+\Psi'(x)^2)\frac{\eta}{2(y-\Psi(x))}+\Psi'(x)\right) \\
&\qquad + \tan^{-1}\left(-(1+\Psi'(x)^2)\frac{\eta}{2(y-\Psi(x))}+\Psi'(x)\right)\\
& \to 0 
\end{align*}
as $(x,y) \to (x_0,\Psi(x_0))$ so there exists $\delta \in (0,\eta/2)$ such that
\[ \int_{\mathbb{R}\setminus B_\eta (x_0)} \widetilde{\mathcal{I}}_1(x,y,t) \, \mathrm{d}t  < \frac{\pi\varepsilon}{4\sup_{[-1,1]} \left| h \sqrt{1+{\Psi'}^2} \right|} \]
for $|(x,y) - (x_0,\Psi(x_0))|< \delta$. The right hand side is well defined since we are assuming $h\neq 0$. When putting these inequalities together, we have
\begin{align*}
& \left| \frac{1}{\pi}\int_{-1}^1 h(t) \sqrt{1+\Psi'(t)^2}\,  \widetilde{\mathcal{I}}_1(x,y,t) \, \mathrm{d}t - h(x_0) \sqrt{1+\Psi'(x_0)^2} \right| \\
&\qquad = \left| \frac{1}{\pi}\int_{-\infty}^\infty \left(\chi_{[-1,1]}(t) h(t)  \sqrt{1+\Psi'(t)^2}- h(x_0) \sqrt{1+\Psi'(x_0)^2}\right)  \widetilde{\mathcal{I}}_1(x,y,t) \, \mathrm{d}t  \right| \\
&\qquad \leq \frac{1}{\pi}\int_{B_\eta(x_0)} \left| h(t) \sqrt{1+\Psi'(t)^2}- h(x_0) \sqrt{1+\Psi'(x_0)^2} \right| \widetilde{\mathcal{I}}_1(x,y,t) \, \mathrm{d}t \\
&\qquad \qquad +  \frac{2\sup_{[-1,1]} \left| h \sqrt{1+{\Psi'}^2} \right|}{\pi} \int_{\mathbb{R}\setminus B_\eta(x_0)} \widetilde{\mathcal{I}}_1(x,y,t) \, \mathrm{d}t \\
&\qquad < \frac{\varepsilon}{2} +\frac{\varepsilon}{2} = \varepsilon
\end{align*}
for $|(x,y) - (x_0,\Psi(x_0))|< \delta$. This proves
\begin{equation}\label{LIMIT2}
\frac{1}{\pi}\int_{-1}^1 h(t) \sqrt{1+\Psi'(t)^2}\,  \widetilde{\mathcal{I}}_1(x,y,t) \, \mathrm{d}t \to h(x_0) \sqrt{1+\Psi'(x_0)^2}
\end{equation}
as $(x,y) \to (x_0,\Psi(x_0))$. Combining the limits \eqref{LIMIT1} and \eqref{LIMIT2} gives us
\begin{equation*}
\frac{1}{\pi}\int_{-1}^1 h(t) \sqrt{1+\Psi'(t)^2} \, \mathcal{I}_1(x,y,t) \, \mathrm{d}t \to h(x_0) \sqrt{1+\Psi'(x_0)^2} 
\end{equation*}
as $(x,y) \to (x_0,\Psi(x_0))$.

We can now look at our second integral $\int_{-1}^1 h(t) \sqrt{1+\Psi'(t)^2} \, \mathcal{I}_2(x,y,t) \, \mathrm{d}t$ as $(x,y) \to (x_0,\Psi(x_0))$. The upper bound
\begin{align*}
\big| \mathcal{I}_2(x,y,t) \big| &= \left| \frac{\Psi(x) - \Psi(t) - \Psi'(x)(x-t)}{(x-t)^2+(y-\Psi(t))^2} \right|\\
&= \frac{\left|\int_0^1 \Psi''((1-\tau)t +\tau x) \tau\, \mathrm{d}\tau (x-t)^2\right|}{(x-t)^2+(y-\Psi(t))^2}\\
&\leq \int_0^1 |\Psi''((1-\tau)t +\tau x) \tau| \, \mathrm{d}\tau \\
&\leq \sup_{[-1,1]} |\Psi''|
\end{align*}
enables us to apply the Dominated Convergence Theorem resulting in
\[ \int_{-1}^1 h(t) \sqrt{1+\Psi'(t)^2} \, \mathcal{I}_2(x,y,t) \, \mathrm{d}t \to \int_{-1}^1 h(t) \sqrt{1+\Psi'(t)^2} \mathcal{I}_2(x_0,\Psi(x_0),t) \, \mathrm{d}t\]
as $(x,y) \to (x_0,\Psi(x_0))$. Altogether we have
\begin{align*}
&\frac{\partial G}{\partial y} (x,y) -\Psi'(x) \frac{\partial G}{\partial x} (x,y) \\
&\qquad = \frac{1}{\pi}\int_{-1}^1 h(t) \sqrt{1+\Psi'(t)^2} \left( \mathcal{I}_1(x,y,t) + \mathcal{I}_2(x,y,t) \right) \, \mathrm{d}t \\
&\qquad \to h(x_0) \sqrt{1+\Psi'(x_0)^2} + \frac{1}{\pi}\int_{-1}^1 h(t) \sqrt{1+\Psi'(t)^2} \mathcal{I}_2(x_0,\Psi(x_0),t) \, \mathrm{d}t
\end{align*}
as $(x,y) \to (x_0,\Psi(x_0))$ and so
\[ \frac{\partial G}{\partial \N} (x, \Psi(x)) = h(x)  + F(x), \]
where
\[ F(x) := \frac{1}{\pi\sqrt{1+\Psi'(x)^2}}\int_{-1}^1 h(t) \sqrt{1+\Psi'(t)^2} \left[\frac{\Psi(x) - \Psi(t) - \Psi'(x)(x-t)}{(x-t)^2+(\Psi(x)-\Psi(t))^2}\right] \, \mathrm{d}t \]
for $x\in(-1,1)$.

We will now attempt to show
\begin{align*}
\sqrt{1+\Psi'(x)^2} \big(\nabla G (x,y) \cdot \T(x)\big) &= \frac{\partial G}{\partial x} (x,y) + \Psi'(x) \frac{\partial G}{\partial y} (x,y)   \\
&=  \frac{1}{\pi} \int_{-1}^1  h(t) \sqrt{1+\Psi'(t)^2} \left[\frac{ (x-t) + \Psi'(x)(y-\Psi(t))}{(x-t)^2+(y-\Psi(t))^2}\right] \, \mathrm{d}t 
\end{align*} 
has a continuous extension up to $\Gamma$. We start by noticing
\begin{align*}
&\frac{(x-t) + \Psi'(x)(y-\Psi(t))}{(x-t)^2+(y-\Psi(t))^2} \\
&\qquad = \frac{(\Psi'(x)-\Psi'(t))(y-\Psi(t))}{(x-t)^2+(y-\Psi(t))^2} - \frac{\partial}{\partial t} \left( \log\sqrt{(x-t)^2 + (y-\Psi(t))^2} \right) \\
&\qquad = \frac{(\Psi'(x)-\Psi'(t))(y-\Psi(x))}{(x-t)^2+(y-\Psi(t))^2} + \frac{(\Psi'(x)-\Psi'(t))(\Psi(x)-\Psi(t))}{(x-t)^2+(y-\Psi(t))^2} \\
&\qquad \qquad - \frac{\partial}{\partial t} \left( \log\sqrt{(x-t)^2 + (y-\Psi(t))^2} \right) 
\end{align*}
and as a result define
\[ \mathcal{I}_3 (x,y,t) := \frac{(\Psi'(x)-\Psi'(t))(y-\Psi(x))}{(x-t)^2+(y-\Psi(t))^2}, \]
\[ \mathcal{I}_4 (x,y,t) := \frac{(\Psi'(x)-\Psi'(t))(\Psi(x)-\Psi(t))}{(x-t)^2+(y-\Psi(t))^2}, \]
and
\[ \mathcal{I}_5 (x,y,t) := \frac{\partial}{\partial t} \left( \log\sqrt{(x-t)^2 + (y-\Psi(t))^2} \right). \]
Since $\mathcal{I}_3 (x,y,t) = (\Psi'(x)-\Psi'(t))\mathcal{I}_1 (x,y,t)$, our previous work shows
\begin{align*}
\int_{-1}^1 h(t) \sqrt{1+\Psi'(t)^2} \, \mathcal{I}_3(x,y,t) \, \mathrm{d}t \to 0
\end{align*}
as $(x,y) \to (x_0,\Psi(x_0))$. For $\int_{-1}^1 h(t) \sqrt{1+\Psi'(t)^2} \, \mathcal{I}_4(x,y,t) \, \mathrm{d}t$ we have 
\begin{align*}
\big| \mathcal{I}_4 (x,y,t) \big| &= \frac{\left|\int_0^1 \Psi''((1-\tau)x +\tau t) \, \mathrm{d}\tau \right| \left| \int_0^1 \Psi'((1-\tau)x +\tau t) \, \mathrm{d}\tau \right| (x-t)^2}{(x-t)^2+(y-\Psi(t))^2} \\
&\leq \int_0^1 |\Psi''((1-\tau)x +\tau t)| \, \mathrm{d}\tau \int_0^1 |\Psi'((1-\tau)x +\tau t)| \, \mathrm{d}\tau \\
&\leq \sup_{[-1,1]} |\Psi''| \sup_{[-1,1]} |\Psi'|
\end{align*}
implying
\[ \int_{-1}^1 h(t) \sqrt{1+\Psi'(t)^2} \, \mathcal{I}_4(x,y,t) \, \mathrm{d}t \to \int_{-1}^1 h(t) \sqrt{1+\Psi'(t)^2} \, \mathcal{I}_4(x_0,\Psi(x_0),t) \, \mathrm{d}t \]
as $(x,y) \to (x_0,\Psi(x_0))$ by the Dominated Convergence Theorem. Lastly, using integration by parts
\begin{align*}
&\int_{-1}^1 h(t) \sqrt{1+\Psi'(t)^2} \, \mathcal{I}_5(x,y,t) \, \mathrm{d}t \\
&\qquad = h(1) \sqrt{1+\Psi'(1)^2} \log\sqrt{(x-1)^2 + (y-\Psi(1))^2} \\
&\qquad \qquad - h(-1) \sqrt{1+\Psi'(-1)^2} \log\sqrt{(x+1)^2 + (y-\Psi(-1))^2} \\
&\qquad \qquad -\int_{-1}^1 \frac{\partial}{\partial t} \left( h(t) \sqrt{1+\Psi'(t)^2} \right) \log\sqrt{(x-t)^2 + (y-\Psi(t))^2} \, \mathrm{d}t,
\end{align*}
which has a limit as $(x,y) \to (x_0,\psi(x_0))$ since the above integral has the same form as the integral given by $G$. Overall, we have shown that
\[ \frac{\partial G}{\partial x} (x,y) + \Psi'(x) \frac{\partial G}{\partial y} (x,y) \]
has a continuous extension to $\Gamma$. Since both
\[ \frac{\partial G}{\partial y} (x,y) -\Psi'(x) \frac{\partial G}{\partial x} (x,y) \qquad \text{and} \qquad \frac{\partial G}{\partial x} (x,y) + \Psi'(x) \frac{\partial G}{\partial y} (x,y) \]
have continuous extensions to $\Gamma$, it implies that $\frac{\partial G}{\partial x}$ and $\frac{\partial G}{\partial y}$ also have continuous extensions to $\Gamma$.

Altogether, we have shown that $G\in C^2(U_1\setminus\overline{\Omega}) \cap C^1(U_1\setminus\Omega)$ solves
\begin{empheq}[left =\empheqlbrace]{alignat*=2}
\Delta G(x,y) &=0 & \qquad &\text{for} \ (x,y)\in U_1\setminus \overline{\Omega}\\
G(x,\Psi(x))&= g(x)-H(x) &&\text{for} \ x\in(-1,1) \\
\frac{\partial G}{\partial \N} (x,\Psi(x)) &= h(x) + F(x) &&\text{for} \ x\in(-1,1). 
\end{empheq}
Now to construct a solution to the Cauchy problem \eqref{CP2}, it is enough to find an open connected neighbourhood $U_2$ of $\Gamma$ and $W\in C^2(U_2\setminus\overline{\Omega}) \cap C^1(U_2\setminus\Omega)$ solving
\begin{empheq}[left =\empheqlbrace]{alignat*=2}
\Delta W(x,y) &=0 & \qquad &\text{for} \ (x,y)\in U_2 \setminus\overline{\Omega}\\
W(x,\Psi(x))&= H(x) &&\text{for} \ x\in(-1,1) \\
\frac{\partial W}{\partial\N} (x,\Psi(x)) &= -F(x) &&\text{for} \ x\in(-1,1) .
\end{empheq}
Note that $F$ is real analytic by the same reasoning we used to justify integral \eqref{ANAINT2} is real analytic. Since $F$ and $H$ are real analytic, the existence of such a $W$ is guaranteed by the Cauchy--Kovalevskaya Theorem. Finally, we have that $u := G+W$ solves the Cauchy problem \eqref{CP2} for $U = U_1\cap U_2$.
\end{proof}

Using this result, we return to the Cauchy problem \eqref{ECP}. As in the introduction, we take $\Omega\subset\mathbb{R}^2$ to be an open bounded and simply connected set with real analytic boundary. We finish this section with a proof of the following theorem relating the existence of solutions to the Cauchy problem \eqref{ECP} to the boundary data regularity.

\begin{thm}\label{CPHV}
Let $f,h \in C^1(\partial\Omega)$. There exists an open connected set $U\subset\mathbb{R}^2$ satisfying $\overline{\Omega}\subset U$ and a vector field $\B \in C^1(U\setminus \overline{\Omega};\mathbb{R}^2) \cap C(U\setminus \Omega;\mathbb{R}^2)$ that solves the Cauchy problem \eqref{ECP} if and only if $f-\mathcal{H} h$ is real analytic on $\partial\Omega$.
\end{thm} 
\begin{proof}
We first prove that $f-\mathcal{H} h$ is necessarily real analytic. Suppose that $\B \in {C^1(U\setminus \overline{\Omega};\mathbb{R}^2)} \cap C(U\setminus \Omega;\mathbb{R}^2)$ solves the Cauchy problem \eqref{ECP}. We will show ${f-\mathcal{H} h}$ is real analytic at $\bm{v}_0\in\partial\Omega$. We start by noticing that there exists an open neighbourhood ${V\subset U}$ of $\bm{v}_0$ such that $V\setminus \Omega$ is simply connected. Since on simply connected domains every harmonic vector field has a harmonic scalar potential, there exists $u\in {C^2(V\setminus \overline{\Omega})} \cap C^1(V\setminus\Omega)$ satisfying $\B =\nabla u$. Furthermore, since every real analytic curve is locally the graph of a real analytic function, there exists $a>0$ and $\Psi\colon[-a,a]\to \mathbb{R}$ such that 
\begin{equation}\label{GAMMA}
\Gamma := \{ \bm{v}_0 + \widetilde{x}\T(\bm{v}_0) + \Psi(\widetilde{x})\bm{n}(\bm{v}_0)  : \widetilde{x}\in (-a,a) \}  
\end{equation} 
is a segment of $\partial\Omega$ containing $\bm{v}_0$, $\Gamma$ lies within $V$, and $\Psi$ is has a real analytic extension to an open neighbourhood of $[-a,a]$. That the vector field $\B$ is a solution to the Cauchy problem \eqref{ECP} implies that $u$ solves
\begin{subequations}\label{SYS2}
\begin{empheq}[left =\empheqlbrace]{alignat=2}
\Delta u &=0 & \quad &\text{in} \  V \setminus\overline{\Omega}\\
\nabla u\cdot \T &= f &&\text{on} \ \Gamma \\
\frac{\partial u}{\partial\N} &= h &&\text{on} \ \Gamma.
\end{empheq}
\end{subequations}

We now perform the coordinate transformation 
\begin{equation}\label{CT}
(x,y)= T(\widetilde{x},\widetilde{y}) := \bm{v}_0 + \widetilde{x}\bm{t}(\bm{v}_0) + \widetilde{y}\bm{n}(\bm{v}_0),
\end{equation} 
with $\widetilde{u}(\widetilde{x},\widetilde{y}) := u(T(\widetilde{x},\widetilde{y}))$ to system \eqref{SYS2}. Note that $T$ is an isometry and the Laplacian is invariant under isometries. $T$ transforms the equation $\nabla u\cdot \T = f$ to 
\begin{equation}\label{FTLI}
\nabla \widetilde{u}(\widetilde{x},\Psi(\widetilde{x})) \cdot (1,\Psi'(\widetilde{x})) = \widetilde{f}(\widetilde{x})\sqrt{1+\Psi'(\widetilde{x})^2}, 
\end{equation} 
where $\widetilde{f}(\widetilde{x}) := f(T(\widetilde{x},\Psi(\widetilde{x})))$. By the Fundamental Theorem of Calculus for Line Integrals, equation \eqref{FTLI} can be integrated to obtain $\widetilde{u}(\widetilde{x},\Psi(\widetilde{x})) = \widetilde{g}(\widetilde{x})$ where $\widetilde{g}'(\widetilde{x}) =\widetilde{f}(\widetilde{x})\sqrt{1+\Psi'(\widetilde{x})^2}$. Overall, by letting 
\[ \widetilde{V}:= T^{-1}V, \quad \widetilde{\Omega}:= T^{-1}\Omega, \quad \widetilde{\N}:=T^{-1}\bm{\N}, \quad \widetilde{h}(\widetilde{x}):= h(T(\widetilde{x},\Psi(\widetilde{x}))), \] 
we have a solution $\widetilde{u}$ to
\begin{subequations}\label{SYS1}
\begin{empheq}[left =\empheqlbrace]{alignat=2}
\Delta \widetilde{u} (\widetilde{x},\widetilde{y}) &=0 & \quad &\text{for} \ (\widetilde{x},\widetilde{y})\in \widetilde{V} \setminus\overline{\widetilde{\Omega}}\\
\widetilde{u}(\widetilde{x},\Psi(\widetilde{x})) &= \widetilde{g}(\widetilde{x}) &&\text{for} \ \widetilde{x}\in(-a,a) \\
\frac{\partial\widetilde{u}}{\partial\widetilde{\N}} (\widetilde{x},\Psi(\widetilde{x})) &= \widetilde{h}(\widetilde{x}) &&\text{for} \ \widetilde{x}\in(-a,a).
\end{empheq}
\end{subequations}

Theorem \ref{HAD2} tells us the existence of a solution to this system implies that
\begin{equation}\label{FUNCH}
H(\widetilde{x}) := \widetilde{g}(\widetilde{x}) - \frac{1}{\pi}\int_{-a}^a \widetilde{h}(t) \sqrt{1+\Psi'(t)^2} \log\sqrt{(\widetilde{x}-t)^2 + (\Psi(\widetilde{x})-\Psi(t))^2} \, \mathrm{d}t 
\end{equation} 
is real analytic on $(-a,a)$. It follows that the derivative of $H$ is also real analytic on $(-a,a)$. We will now show that in a Cauchy principle value sense we can interchange the integral in $H$ with a derivative. Care has to be taken around the singularity of the integrand. For small positive $\varepsilon$ we define $I_\varepsilon(\widetilde{x}) = (\widetilde{x}-\varepsilon,\widetilde{x}+\varepsilon)$ and 
\[ J_\varepsilon (\widetilde{x}) := \int_{[-a,a]\setminus I_\varepsilon(\widetilde{x})} \widetilde{h}(t) \sqrt{1+\Psi'(t)^2} \log\sqrt{(\widetilde{x}-t)^2 + (\Psi(\widetilde{x})-\Psi(t))^2} \, \mathrm{d}t, \]
which satisfies $H(\widetilde{x}) = \widetilde{g}(\widetilde{x}) - \frac{1}{\pi}\lim_{\varepsilon\to 0}J_\varepsilon(\widetilde{x})$. Using the Leibniz integral rule we have
\begin{align*}
J_\varepsilon' (\widetilde{x}) = \int_{[-a,a]\setminus I_\varepsilon(\widetilde{x})}\mathcal{I}(\widetilde{x},t) \, \mathrm{d}t + R_\varepsilon(\widetilde{x}),
\end{align*} 
where
\[ \mathcal{I}(\widetilde{x},t) := \widetilde{h}(t) \sqrt{1+\Psi'(t)^2} \left[ \frac{\widetilde{x}-t+\Psi'(\widetilde{x})(\Psi(\widetilde{x})-\Psi(t))}{(\widetilde{x}-t)^2 + (\Psi(\widetilde{x})-\Psi(t))^2} \right] \]
and
\begin{align*}
R_\varepsilon (\widetilde{x}) &:= \widetilde{h}(\widetilde{x}-\varepsilon) \sqrt{1+\Psi'(\widetilde{x}-\varepsilon)^2} \log(\sqrt{\varepsilon^2+(\Psi(\widetilde{x})-\Psi(\widetilde{x}-\varepsilon))^2}) \\
&\qquad -\widetilde{h}(\widetilde{x}+\varepsilon) \sqrt{1+\Psi'(\widetilde{x}+\varepsilon)^2} \log(\sqrt{\varepsilon^2+(\Psi(\widetilde{x})-\Psi(\widetilde{x}+\varepsilon))^2}). 
\end{align*} 

Our next step is to prove $J_\varepsilon'$ converges uniformly as $\varepsilon$ goes to zero. Firstly, $|R_\varepsilon(\widetilde{x})| \lesssim \varepsilon|\log\varepsilon| + \varepsilon$ and so $R_\varepsilon \to 0$ uniformly as $\varepsilon \to 0$. Secondly, observe that
\[ \mathcal{I}(\widetilde{x},t) = \frac{\psi(\widetilde{x},t)}{\widetilde{x}-t}, \]
where
\[ \psi(\widetilde{x},t) := \widetilde{h}(t)\sqrt{1+\Psi'(t)^2} \left[ \frac{1+\Psi'(\widetilde{x})\phi(\widetilde{x},t)}{1 + \phi(\widetilde{x},t)^2} \right] \]
and $\phi\in C^2([-a,a]^2)$ satisfies 
\[ \Psi(\widetilde{x})- \Psi(t) = \phi(\widetilde{x},t)(\widetilde{x}-t). \]
Note that $\psi\in C^1([-a,a]^2)$ and so by letting $0<\delta<\varepsilon$ we have
\begin{align*}
&\left| \int_{[-a,a]\setminus I_\varepsilon(\widetilde{x})}\mathcal{I}(\widetilde{x},t) \, \mathrm{d}t -\int_{[-a,a]\setminus I_\delta(\widetilde{x})}\mathcal{I}(\widetilde{x},t) \, \mathrm{d}t \right| \\ 
&\qquad = \left| \int_{ I_\varepsilon(\widetilde{x}) \setminus I_\delta(\widetilde{x})} \frac{\psi(\widetilde{x},t)}{\widetilde{x}-t} \, \mathrm{d}t \right| \\
&\qquad = \left| \int_{ I_\varepsilon(\widetilde{x}) \setminus I_\delta(\widetilde{x})} \frac{\psi(\widetilde{x},t)}{\widetilde{x}-t} \, \mathrm{d}t - \psi(\widetilde{x},\widetilde{x}) \int_{ I_\varepsilon(\widetilde{x}) \setminus I_\delta(\widetilde{x})} \frac{1}{\widetilde{x}-t} \, \mathrm{d}t \right| \\
&\qquad \leq  \int_{ I_\varepsilon(\widetilde{x}) \setminus I_\delta(\widetilde{x})} \left|\frac{\psi(\widetilde{x},\widetilde{x})-\psi(\widetilde{x},t)}{\widetilde{x}-t}\right| \, \mathrm{d}t  \\
&\qquad \leq \int_{ I_\varepsilon(\widetilde{x}) \setminus I_\delta(\widetilde{x})} \, \sup_{[-a,a]^2} |\partial_2 \psi |  \, \mathrm{d}t \\
&\qquad \leq \varepsilon \sup_{[-a,a]^2} |\partial_2 \psi |.
\end{align*}
This proves that $\int_{[-a,a]\setminus I_\varepsilon(\widetilde{x})}\mathcal{I}(\widetilde{x},t) \, \mathrm{d}t$ is uniformly Cauchy and so converges uniformly as $\varepsilon\to 0$. Therefore, we have shown $J_\varepsilon'$ converges uniformly to ${\lim_{\varepsilon\to 0}\int_{[-a,a]\setminus I_\varepsilon(\widetilde{x})}\mathcal{I}(\widetilde{x},t) \, \mathrm{d}t}$ as $\varepsilon\to 0$ which in turn proves
\[ H'(\widetilde{x}) = \widetilde{g}'(\widetilde{x}) - \frac{1}{\pi}\lim_{\varepsilon\to 0}\int_{[-a,a]\setminus I_\varepsilon(\widetilde{x})}\mathcal{I}(\widetilde{x},t) \, \mathrm{d}t. \]

We will show that we can replace $I_\varepsilon(\widetilde{x})$ with the set 
\[ S_\varepsilon(\widetilde{x}) := \{ t\in[-a,a] : |(\widetilde{x},\Psi(\widetilde{x})) - (t,\Psi(t))|< \varepsilon \} \]
to make the integral independent on the parameterisation of $\Gamma$. This independence will come from the fact that 
\[ \Gamma\cap B_\varepsilon (T(\widetilde{x},\Psi(\widetilde{x}))) = \{T(t,\Psi(t))\in\mathbb{R}^2 : t\in S_\varepsilon(\widetilde{x}) \}. \]
Note that if $t\in S_\varepsilon(\widetilde{x})$, then
\[ (\widetilde{x}-t)^2 + (\Psi(\widetilde{x})-\Psi(t))^2 < \varepsilon^2, \]
which implies
\[ |\widetilde{x}-t| < \frac{\varepsilon}{\sqrt{1+\phi(\widetilde{x},t)^2}}. \]
For ease of notation it will be useful to define $\varphi\in C^2([-a,a]^2)$ by
\[ \varphi(\widetilde{x},t) := \frac{1}{\sqrt{1+\phi(\widetilde{x},t)^2}}. \]
We now wish to show,
\begin{align*}
\lim_{\varepsilon\to 0} \int_{[-a,a]\setminus I_\varepsilon(\widetilde{x})}\mathcal{I}(\widetilde{x},t) \, \mathrm{d}t = \lim_{\varepsilon\to 0} \int_{[-a,a]\setminus S_\varepsilon(\widetilde{x})}\mathcal{I}(\widetilde{x},t) \, \mathrm{d}t. 
\end{align*}
Fix $\zeta>0$. Let $0<\eta<\varphi(\widetilde{x},\widetilde{x})$ be small enough such that
\[ \log\left(\frac{\varphi(\widetilde{x},\widetilde{x})+\eta}{\varphi(\widetilde{x},\widetilde{x})-\eta} \right) < \frac{\zeta}{2\sup_{[-a,a]^2}|\psi|}. \]
Since $\varphi$ is continuous there exists $\varepsilon>0$ small enough such that if $t\in [-a,a]$ and $|\widetilde{x}-t|<\varepsilon$, then $|\varphi(\widetilde{x},t) - \varphi(\widetilde{x},\widetilde{x})|<\eta$. Now that we have chosen an $\varepsilon$, let $t\in S_\varepsilon(\widetilde{x})$. It follows that $|\widetilde{x}-t|<\varepsilon$ and $|\widetilde{x}-t|<\varepsilon \varphi(\widetilde{x},t)$. From this it is evident that 
\[ I_{\varepsilon(\varphi(\widetilde{x},\widetilde{x})-\eta)}(\widetilde{x}) \subset S_\varepsilon(\widetilde{x}) \subset I_{\varepsilon(\varphi(\widetilde{x},\widetilde{x})+\eta)}(\widetilde{x}), \]
and obviously 
\[ I_{\varepsilon(\varphi(\widetilde{x},\widetilde{x})-\eta)}(\widetilde{x}) \subset I_{\varepsilon\varphi(\widetilde{x},\widetilde{x})}(\widetilde{x}) \subset I_{\varepsilon(\varphi(\widetilde{x},\widetilde{x})+\eta)}(\widetilde{x}). \]
These inclusions guarantee  
\begin{align*}
&\left| \int_{[-a,a]\setminus I_{\varepsilon\varphi(\widetilde{x},\widetilde{x})}(\widetilde{x})}\mathcal{I}(\widetilde{x},t) \, \mathrm{d}t - \int_{[-a,a]\setminus S_\varepsilon(\widetilde{x})}\mathcal{I}(\widetilde{x},t) \, \mathrm{d}t \right| \\
&\qquad = \left| \int_{S_\varepsilon(\widetilde{x}) \setminus I_{\varepsilon\varphi(\widetilde{x},\widetilde{x})}(\widetilde{x})}\mathcal{I}(\widetilde{x},t) \, \mathrm{d}t - \int_{ I_{\varepsilon\varphi(\widetilde{x},\widetilde{x})}(\widetilde{x}) \setminus S_\varepsilon(\widetilde{x})}\mathcal{I}(\widetilde{x},t) \, \mathrm{d}t \right| \\
&\qquad \leq \int_{S_\varepsilon(\widetilde{x}) \setminus I_{\varepsilon\varphi(\widetilde{x},\widetilde{x})}(\widetilde{x})} |\mathcal{I}(\widetilde{x},t)| \, \mathrm{d}t + \int_{ I_{\varepsilon\varphi(\widetilde{x},\widetilde{x})}(\widetilde{x}) \setminus S_\varepsilon(\widetilde{x})}|\mathcal{I}(\widetilde{x},t)| \, \mathrm{d}t \\
&\qquad \leq \int_{I_{\varepsilon(\varphi(\widetilde{x},\widetilde{x})+\eta)}(\widetilde{x}) \setminus I_{\varepsilon(\varphi(\widetilde{x},\widetilde{x})-\eta)}(\widetilde{x})} |\mathcal{I}(\widetilde{x},t)| \, \mathrm{d}t \\
&\qquad \leq \sup_{[-a,a]^2}|\psi|\int_{I_{\varepsilon(\varphi(\widetilde{x},\widetilde{x})+\eta)}(\widetilde{x}) \setminus I_{\varepsilon(\varphi(\widetilde{x},\widetilde{x})-\eta)}(\widetilde{x})} \left|\frac{1}{\widetilde{x}-t}\right| \, \mathrm{d}t \\
&\qquad = 2\sup_{[-a,a]^2}|\psi| \int_{\varepsilon(\varphi(\widetilde{x},\widetilde{x})-\eta)}^{\varepsilon(\varphi(\widetilde{x},\widetilde{x})+\eta)} \frac{1}{t} \, \mathrm{d}t \\
&\qquad = 2\sup_{[-a,a]^2}|\psi| \log\left(\frac{\varphi(\widetilde{x},\widetilde{x})+\eta}{\varphi(\widetilde{x},\widetilde{x})-\eta} \right) \\
&\qquad <\zeta
\end{align*}
as required. We can conclude that
\[ H'(\widetilde{x}) = \widetilde{g}'(\widetilde{x}) - \frac{1}{\pi}\lim_{\varepsilon\to 0}\int_{[-a,a]\setminus S_\varepsilon(\widetilde{x})}\mathcal{I}(\widetilde{x},t) \, \mathrm{d}t. \]

Observe that we can rewrite $\mathcal{I}(\widetilde{x},t)/\sqrt{1+\Psi'(\widetilde{x})^2}$ as
\begin{align*}
\frac{\mathcal{I}(\widetilde{x},t)}{\sqrt{1+\Psi'(\widetilde{x})^2}} &= \widetilde{h}(t) |(t,\Psi'(t))| \left[ \frac{(1,\Psi'(\widetilde{x}))\cdot ((\widetilde{x},\Psi(\widetilde{x}))-(t,\Psi(t)))}{|(1,\Psi'(\widetilde{x}))||(\widetilde{x},\Psi(\widetilde{x}))-(t,\Psi(t))|^2} \right] \\
&= h(T(t,\Psi(t)))|T(t,\Psi'(t))| \left[ \frac{T(1,\Psi'(\widetilde{x}))\cdot (T(\widetilde{x},\Psi(\widetilde{x}))-T(t,\Psi(t)))}{|T(1,\Psi'(\widetilde{x}))||T(\widetilde{x},\Psi(\widetilde{x}))-T(t,\Psi(t))|^2} \right].
\end{align*} 
Recall that $T(\widetilde{x},\Psi(\widetilde{x}))$ is a parameterisation of $\Gamma$. Hence we define the parameterisation $\bm{\gamma}(s) := T(s,\Psi(s))$ for $s\in[-a,a]$. Since $T$ is an isometry, it holds that $\bm{\gamma}'(s) = T(1,\Psi'(s))$. Therefore,
\begin{align*} \frac{\mathcal{I}(s,t)}{\sqrt{1+\Psi'(s)^2}} &= h(\bm{\gamma}(t))|\bm{\gamma}'(t)| \left[ \frac{\bm{\gamma}'(s)\cdot (\bm{\gamma}(s)-\bm{\gamma}(t))}{|\bm{\gamma}'(s)||\bm{\gamma}(s)-\bm{\gamma}(t)|^2} \right] \\
&= h(\bm{\gamma}(t))|\bm{\gamma}'(t)| \left[ \frac{\bm{t}(\bm{\gamma}(s))\cdot (\bm{\gamma}(s)-\bm{\gamma}(t))}{|\bm{\gamma}(s)-\bm{\gamma}(t)|^2} \right]. 
\end{align*}
Furthermore, by recalling the formula for $\widetilde{g}'$, it follows that 
\[  \widetilde{g}'(s) = \widetilde{f}(s)\sqrt{1+\Psi'(s)^2}  = f(\bm{\gamma}(s))\sqrt{1+\Psi'(s)^2} \]
By substituting these expressions into our equation for $H'$ gives us
\begin{align*}
\frac{H'(s)}{\sqrt{1+\Psi'(s)^2}} &= f(\bm{\gamma}(s)) - \frac{1}{\pi}\lim_{\varepsilon\to 0}\int_{[-a,a]\setminus S_\varepsilon(s)} h(\bm{\gamma}(t))|\bm{\gamma}'(t)| \left[ \frac{\bm{t}(\bm{\gamma}(s))\cdot (\bm{\gamma}(s)-\bm{\gamma}(t))}{|\bm{\gamma}(s)-\bm{\gamma}(t)|^2} \right] \, \mathrm{d}t \\
&= f(\bm{\gamma}(s)) - \frac{1}{\pi}\lim_{\varepsilon\to 0}\int_{\Gamma\setminus B_\varepsilon(\bm{\gamma}(s))} h(\bm{w}) \left[ \frac{\bm{t}(\bm{\gamma}(s))\cdot (\bm{\gamma}(s)-\bm{w})}{|\bm{\gamma}(s)-\bm{w}|^2} \right] \, \mathrm{d}\bm{w}.
\end{align*}
Since the left-hand side is real analytic on $(-a,a)$, and $\bm{\gamma}(0)=\bm{v}_0$, it follows by definition that  
\begin{equation}\label{PART1}
 f(\bm{v}) - \frac{1}{\pi}\lim_{\varepsilon\to 0}\int_{\Gamma\setminus B_\varepsilon(\bm{v})} h(\bm{w}) \frac{\bm{t}(\bm{v})\cdot (\bm{v}-\bm{w})}{|\bm{v}-\bm{w}|^2}  \, \mathrm{d}\bm{w}
\end{equation}
is real analytic at $\bm{v} = \bm{v}_0$.

It remains to show 
\begin{equation}\label{PART2}
 \int_{\partial\Omega \setminus \Gamma} h(\bm{w}) \frac{\bm{t}(\bm{v})\cdot (\bm{v}-\bm{w})}{|\bm{v}-\bm{w}|^2}  \, \mathrm{d}\bm{w} 
\end{equation}
is real analytic at $\bm{v}=\bm{v}_0$. Let $\bm{\sigma}\colon I \to\mathbb{R}^2$ be a real analytic parameterisation of $\partial\Omega\setminus\Gamma$. Since there is a positive distance between $\partial\Omega\setminus\Gamma$ and $\bm{v_0}$, there exists $r>0$ small enough such that $\bm{\gamma}$ has a complex analytic extension to the complex disc $D_r(0)$ and $|\bm{\gamma}(z) - \bm{\sigma}(t)|>0$ for all $z\in D_r(0)$ and $t\in I$. Consequently, 
\[ h(\bm{\sigma}(t))|\bm{\sigma}'(t)| \left[ \frac{\bm{t}(\bm{\gamma}(z))\cdot (\bm{\gamma}(z)-\bm{\sigma}(t))}{|\bm{\gamma}(z)-\bm{\sigma}(t)|^2} \right] \]
is complex analytic on $D_r(0)$  for all $t\in I$. We can now apply Lemma \ref{ANAINT} to justify 
\[ \int_{I} h(\bm{\sigma}(t))|\bm{\sigma}'(t)| \left[ \frac{\bm{t}(\bm{\gamma}(z))\cdot (\bm{\gamma}(z)-\bm{\sigma}(t))}{|\bm{\gamma}(z)-\bm{\sigma}(t)|^2} \right] \, \mathrm{d}t \]
being complex analytic on $D_r(0)$. Thus the integral \eqref{PART2} is real analytic at $\bm{v}=\bm{v}_0$.

By combining the real analyticity of expressions \eqref{PART1} and \eqref{PART2}, we obtain the real analyticity of $f-\mathcal{H}h$ at $\bm{v}_0$. As $\bm{v}_0$ was chosen arbitrarily, $f-\mathcal{H}h$ is real analytic on $\partial\Omega$. This concludes the necessity section of the proof. \bigbreak

For the sufficiency section of the proof, we start by assuming $f-\mathcal{H}h$ is real analytic on $\partial\Omega$. We begin by reversing the arguments used in the necessity part. Afterwards, we will have to make sure solutions over different regions coincide on their overlap. 

Given $\bm{v}_0\in\partial\Omega$, let the boundary segment $\Gamma$ in \eqref{GAMMA} and coordinate transform $T$ in \eqref{CT} be defined as before. By reversing previous arguments, the function $f-\mathcal{H}h$ being real analytic on $\Gamma$ implies $H$, defined in \eqref{FUNCH}, is real analytic on $(-a,a)$. Therefore, by Theorem \ref{HAD2} there exists $\widetilde{V}\subset\mathbb{R}^2$, an open neighbourhood of $T\Gamma$, and $\widetilde{u}\in C^2(\widetilde{V}\setminus \overline{\Omega}) \cap C^1(\widetilde{V}\setminus \Omega)$ that solves the scalar system \eqref{SYS1}. If we now perform the coordinate transform $T^{-1}$ on system \eqref{SYS1}, then $u(x,y):= \widetilde{u}(T^{-1}(x,y))$ solves system \eqref{SYS2} with $V:=T\widetilde{V}$. Therefore, the vector field $\B := \nabla u$ satisfies $\B \in C^1(V\setminus \overline{\Omega};\mathbb{R}^2) \cap C(V\setminus \Omega;\mathbb{R}^2)$ and solves
\begin{subequations}\label{SYS3}
\begin{empheq}[left=\empheqlbrace]{alignat=2}
\operatorname{div} \B &= 0  \quad && \text{in} \ V\setminus \overline{\Omega} \\
\operatorname{curl} \B &= 0 && \text{in} \ V\setminus \overline{\Omega} \\
\B\cdot\T &= f  && \text{on} \ \Gamma \\
\B\cdot\N &= h && \text{on} \ \Gamma. 
\end{empheq}
\end{subequations}

For all $\bm{v}\in\partial\Omega$ we can find a boundary segment $\Gamma_{\bm{v}}\subset\partial\Omega$ that is the graph of a real analytic function and contains $\bm{v}$. We can apply the above method to obtain, for every $\bm{v}\in\partial\Omega$, an open neighbourhood $V_{\bm{v}}\subset \mathbb{R}^2$ of $\Gamma_{\bm{v}}$, and vector field $\B_{\bm{v}} \in C^1(V_{\bm{v}}\setminus \overline{\Omega};\mathbb{R}^2) \cap C(V_{\bm{v}}\setminus \Omega;\mathbb{R}^2)$ solving system \eqref{SYS3} with $V=V_{\bm{v}}$ and $\Gamma=\Gamma_{\bm{v}}$. To show that the $\{\B_{\bm{v}}\}_{\bm{v}\in\partial\Omega}$ can be combined to form a solution to the Cauchy problem \eqref{ECP}, we need to make sure the $\B_{\bm{v}}$ coincide on the regions where they overlap. We will do this by restricting our vector fields to regions which we call exterior collar neighbourhoods.  

Given boundary segment $\Gamma\subseteq\partial\Omega$ and continuous function $l \colon\Gamma\to(0,\infty]$, we define the fibre $F(\bm{w}) := \{\bm{w}+ \varepsilon\N(\bm{w}) : \varepsilon\in[0,l(\bm{w}))\}$ for $\bm{w}\in\Gamma$. If the collection of fibres $\{F(\bm{w})\}_{\bm{w}\in\Gamma}$ is pairwise disjoint, we call $N:= \bigcup_{\bm{w}\in\Gamma} F(\bm{w})$ an exterior collar neighbourhood of $\Gamma$. We also say $N$ has width $l(\bm{w})$ at $\bm{w}\in\Gamma$. An example of an exterior collar neighbourhood is given by the shaded region in Figure \ref{ETN}. Note, the existence of an exterior collar neighbourhood of a curve is guaranteed if the curve is $C^2$. Furthermore, as $\partial\Omega$ is compact and sufficiently regular there exists a constant $l^*\in(0,\infty]$ such that ${N^* :=  \bigcup_{\bm{w}\in\partial\Omega} \{\bm{w}+ \varepsilon\N(\bm{w}) : \varepsilon\in[0,l^*)\}}$ is an exterior collar neighbourhood of $\partial\Omega$ with constant width $l^*$.

\begin{figure}[H]
\begin{tikzpicture}[scale=3, rotate around={-40:(0,0)}]

\draw [thick, dashed, fill=gray!20] plot [smooth, samples=100, domain=-1:1] ({\x-0.3*3*(\x)^2*(1-(\x)^4)}, {(\x)^3+0.3*(1-(\x)^4)}) -- plot [smooth, samples=100, domain=1:-1] (\x, {(\x)^3});
\draw [thick] plot [smooth, samples=100, domain=-1.1:1.1] (\x, {(\x)^3}) node [right] {$\partial\Omega$};

\foreach \x in {-1,-0.95,-0.9,-0.85,-0.75,-0.7,...,1}
\draw [thick] (\x, {(\x)^3}) -- ({\x-0.3*3*(\x)^2*(1-(\x)^4)}, {(\x)^3+0.3*(1-(\x)^4)});

\node [left, red] at ({-0.8-0.3*3*(-0.8)^2*(1-(-0.8)^4)}, {(-0.8)^3+0.3*(1-(-0.8)^4)}) {$F(\bm{w})$};
\draw [thick, red] (-0.8,-0.8^3) --  ({-0.8-0.3*3*(-0.8)^2*(1-(-0.8)^4)}, {(-0.8)^3+0.3*(1-(-0.8)^4)});
\draw [fill] (-0.8,-0.8^3) circle [radius=0.6pt] node [below] {$\bm{w}$};
\node [below left] at (0,0) {$\Gamma$};
\node [fill=gray!20] at (-0.2,0.14) {$N$};

\end{tikzpicture}
\centering
\caption{Exterior collar neighbourhood $N$.}
\label{ETN}
\end{figure}
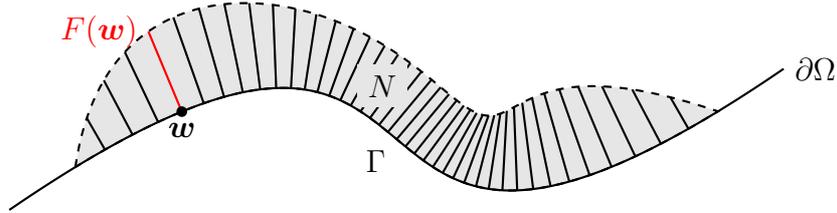

For each $\bm{v}\in\partial\Omega$ it is easy to construct an exterior collar neighbourhood $N_{\bm{v}}$ of $\Gamma_{\bm{v}}$ that is contained within $V_{\bm{v}}\cap N^*$. We restrict the local solutions $\B_{\bm{v}}$ to $N_{\bm{v}}$ in order to avoid overlaps where the $\B_{\bm{v}}$ do not coincide. Choosing the $N_{\bm{v}}$ to be within $N^*$ guarantees that for distinct $\bm{v},\bm{w}\in\partial\Omega$, the intersection $N_{\bm{v}}\cap N_{\bm{w}}$ is connected and in particular an exterior collar neighbourhood of $\Gamma_{\bm{v}}\cap\Gamma_{\bm{w}}$. This is trivially satisfied if $N_{\bm{v}}\cap N_{\bm{w}}$ and $\Gamma_{\bm{v}}\cap\Gamma_{\bm{w}}$ are empty. 

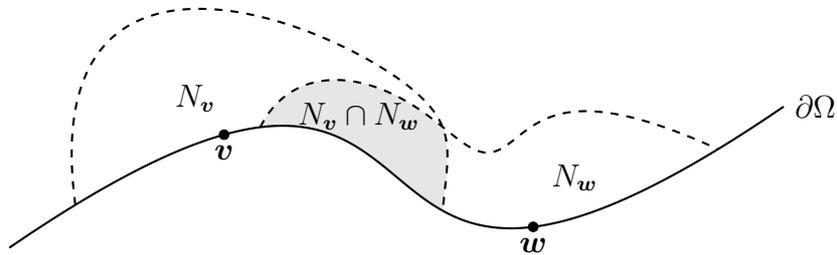
\begin{figure}[H]
\begin{tikzpicture}[scale=3, rotate around={-40:(0,0)}]

\begin{scope}
\clip  plot [smooth, samples=100, domain=-1:0.25] ({\x-0.35*3*(\x)^2*(1-(1.6*\x+0.6)^4)}, {(\x)^3+0.35*(1-(1.6*\x+0.6)^4)}) -- plot [smooth, samples=100, domain=0.25:-1] (\x, {(\x)^3});
\clip plot [smooth, samples=100, domain=-0.6:1] ({\x-0.3*3*(\x)^2*(1-(1.25*\x-0.25)^4)}, {(\x)^3+0.3*(1-(1.25*\x-0.25)^4)}) -- plot [smooth, samples=100, domain=1:-0.6] (\x, {(\x)^3});
\fill [gray!20] (0,0) circle [radius=1];
\end{scope}

\draw [thick] plot [smooth, samples=100, domain=-1.1:1.1] (\x, {(\x)^3}) node [right] {$\partial\Omega$};
\draw [thick, dashed] plot [smooth, samples=100, domain=-1:0.25] ({\x-0.35*3*(\x)^2*(1-(1.6*\x+0.6)^4)}, {(\x)^3+0.35*(1-(1.6*\x+0.6)^4)});
\draw [thick, dashed] plot [smooth, samples=100, domain=-0.6:1] ({\x-0.3*3*(\x)^2*(1-(1.25*\x-0.25)^4)}, {(\x)^3+0.3*(1-(1.25*\x-0.25)^4)});

\node at (-0.3,0.1) {$N_{\bm{v}}\cap N_{\bm{w}}$};
\node at (-0.9,-0.3) {$N_{\bm{v}}$};
\node at (0.6,0.5) {$N_{\bm{w}}$};
\draw [fill] (-0.7,-0.7^3) circle [radius=0.6pt] node [below] {$\bm{v}$};
\draw [fill] (0.6,0.6^3) circle [radius=0.6pt] node [below] {$\bm{w}$};

\end{tikzpicture}
\centering
\caption{Intersection of exterior collar neighbourhoods $N_{\bm{v}}$ and $N_{\bm{w}}$.}
\label{ETN2}
\end{figure}

Now $\B_{\bm{v}}$ and $\B_{\bm{w}}$ solve system \eqref{SYS3} with $V\setminus\Omega=N_{\bm{v}}\cap N_{\bm{w}}$ and $\Gamma=\Gamma_{\bm{v}}\cap\Gamma_{\bm{w}}$. Holmgren's Uniqueness Theorem \cite[\S2.3]{ReRo} tells us that for open connected neighbourhoods $V$ of $\Gamma$, solutions to system \eqref{SYS3} are unique. Therefore, $\B_{\bm{v}}$ and $\B_{\bm{w}}$ must coincide on $N_{\bm{v}}\cap N_{\bm{w}}$. Hence, the vector field $\B^*$, defined pointwise by $\B^*(\bm{w}) := \B_{\bm{v}}(\bm{w})$ for $\bm{w}\in N_{\bm{v}}$, is well defined on the exterior collar neighbourhood $U^*:=\bigcup_{\bm{v}\in\partial\Omega} N_{\bm{v}}$ of $\partial\Omega$. The vector field $\B^*$ also solves system \eqref{SYS3} with $V=U^*\cup \Omega$ and $\Gamma=\partial\Omega$. As a result, $U=U^*\cup\Omega$ and $\B^*$ solve the Cauchy problem \eqref{ECP}.
\end{proof}

\section{Cauchy--Kovalevskaya Theorem and Distance from Boundary}\label{SEC3}

We recall from the introduction that the equations $\operatorname{div} \B = \operatorname{curl} \B=0$ can be viewed as the Cauchy--Riemann equations of $\mathcal{B} = B_1 - iB_2$ with respect to $z=x+iy$. Furthermore, the equations can be combined into the single complex equation
\[ \frac{\partial\mathcal{B}}{\partial y} = i \frac{\partial\mathcal{B}}{\partial x}. \]
We think of $\mathcal{B}$ as a function from $\mathbb{R}^2$ to $\mathbb{C}$. Let $\bm{\gamma}:\mathbb{T}\to\partial\Omega$ be a real analytic parameterisation for $\partial\Omega$, oriented such that $\bm{\gamma}'^\perp = (-\gamma_2',\gamma_1')$ is an outward normal to $\Omega$. The boundary conditions $\B\cdot\T = f$ and $\B\cdot\N = h$ are equivalent to $\mathcal{B}(\bm{\gamma}(t)) = \Theta(t)$ where
\begin{equation}\label{THETA}
\Theta(t) := (\gamma_1'(t)-i\gamma_2'(t))\left(\frac{f(\bm{\gamma}(t))-ih(\bm{\gamma}(t))}{|\bm{\gamma}'(t)|}\right)
\end{equation} 
and $|\bm{\gamma}'(t)| = \sqrt{\gamma_1'(t)^2+\gamma_2'(t)^2}$.

We can therefore rewrite the Cauchy problem \eqref{ECP} as
\begin{subequations}\label{SYS4}
\begin{empheq}[left =\empheqlbrace]{align}
&\frac{\partial\mathcal{B}}{\partial y} = i \frac{\partial\mathcal{B}}{\partial x} \label{CR}\\
&\mathcal{B}(\bm{\gamma}) = \Theta,
\end{empheq}
\end{subequations}
where it is understood that equation \eqref{CR} is being solved in a neighbourhood of $\partial\Omega\subset\mathbb{R}^2$. Given $f,h\in C^\omega(\partial\Omega)$, we can solve this system using the Cauchy--Kovalevskaya Theorem. Note that the Cauchy--Kovalevskaya Theorem solves the system on both sides of $\partial\Omega$ simultaneously. Unfortunately, this means any result provided by Cauchy--Kovalevskaya Theorem on the external distance at which we can solve, may be affected by singularities that arise on the inside of $\Omega$. 

We would like to gain quantitative information on the size of the region on which we can solve the first-order system \eqref{SYS4}. To do this, we will follow the proof of the Cauchy--Kovalevskaya Theorem given in \cite[\S2.2]{ReRo}, and then find the domains where the Taylor series converges. We will focus our attention on finding how far we can solve system \eqref{SYS4} from the boundary point $\bm{\gamma}(t_0)$ for some arbitrary $t_0\in\mathbb{T}$. Then, we use the same procedure in patching together solutions as in the end of the proof of Theorem \ref{CPHV}.

The first step is to transform our system so that the boundary is flat. We start by considering the variables $(\widetilde{x},\widetilde{y})$ defined according to 
\[ (x,y) = \bm{\gamma}(\widetilde{x}) + \widetilde{y}\bm{\gamma}'^\perp(t_0). \]
This change of variables from $(x,y)$ to $(\widetilde{x},\widetilde{y})$ has the effect of flattening the boundary since the curve $\bm{\gamma}$ is mapped to the line $\widetilde{y}=0$. Furthermore, the variables can be described as follows: As $\widetilde{x}$ varies we travel along the curve $\bm{\gamma}$ whereas as $\widetilde{y}$ varies we travel in the fixed direction $\bm{\gamma}'^\perp(t_0)$, not in the normal direction to the curve. By changing our variables to $(\widetilde{x},\widetilde{y})$, system \eqref{SYS4} becomes
\begin{subequations}
\begin{empheq}[left =\empheqlbrace]{align}
&\frac{\partial\mathcal{B}}{\partial\widetilde{y}} =  i\left(\frac{\gamma_1'(t_0)+i\gamma_2'(t_0)}{\gamma_1'+i\gamma_2'}\right) \frac{\partial\mathcal{B}}{\partial\widetilde{x}} \label{PDE}\\ 
&\mathcal{B}(\widetilde{x},0) = \Theta(\widetilde{x}).
\end{empheq}
\end{subequations}
This change of variables is well defined since there exists $\delta>0$ such that it is a diffeomorphism on $(\widetilde{x},\widetilde{y})\in (t_0-\delta,t_0+\delta)\times\mathbb{R}$. It is important to note that our change of variables has been chosen such that the coefficient in the partial differential equation \eqref{PDE} does not depend on $\widetilde{y}$. 

Now by setting 
\[ \phi(\widetilde{x},\widetilde{y}) = \mathcal{B}(\widetilde{x},\widetilde{y}) - \Theta(\widetilde{x}), \] 
we obtain
\begin{subequations}\label{CBS2}
\begin{empheq}[left =\empheqlbrace]{align}
&\frac{\partial\phi}{\partial\widetilde{y}} =  \Lambda\frac{\partial\phi}{\partial\widetilde{x}} + \Lambda\Theta' \label{FOS} \\ 
&\phi(\widetilde{x},0) = 0, \label{IC}
\end{empheq}
\end{subequations}
where
\begin{equation}\label{LAMBDA}
\Lambda(\widetilde{x}) := i\left(\frac{\gamma_1'(t_0)+i\gamma_2'(t_0)}{\gamma_1'(\widetilde{x})+i\gamma_2'(\widetilde{x})}\right). 
\end{equation} 
By rewriting the system in this form, we can determine all the partial derivatives of $\phi$ at $(t_0,0)$ in terms of the derivatives of $\Lambda$ and $\Lambda\Theta'$ at $t_0$. The derivatives with respect to $\widetilde{x}$ are zero due to the boundary condition \eqref{IC}, and the mixed derivatives can be deduced from equation \eqref{FOS} to be polynomials with positive coefficients and whose variables are the derivatives of $\Lambda$ and $\Lambda\Theta'$ at $t_0$. Therefore, a Taylor series for $\phi$ in terms of $\widetilde{x}$ and $\widetilde{y}$ which is based at $(t_0,0)$ can be constructed. If it can be shown that this Taylor series converges, then it solves system \eqref{CBS2} within its domain of convergence. This is achieved by replacing $\Lambda$ and $\Lambda\Theta'$ in equation \eqref{FOS} with functions whose derivatives at $t_0$ have a larger magnitude, and then showing that this new system has an explicit solution with convergent Taylor series.  

Let $R_1(t_0), R_2(t_0)>0$ be the radii of convergence for the Taylor series 
\[ \Lambda(\widetilde{x}) =  \sum_{n=0}^\infty b_n(t_0) (\widetilde{x}-t_0)^n \qquad \text{and} \qquad \Lambda(\widetilde{x})\Theta'(\widetilde{x}) =  \sum_{n=0}^\infty c_n(t_0) (\widetilde{x}-t_0)^n \]
respectively. Note that $b_0 =\Lambda(t_0)=i$. For $r\in (0,\min\{R_1,R_2\})$ let
\[ M_1(r) := \sup\{1, |b_1|r, |b_2|r^2, \ldots \} \]
and
\[ M_2(r) := \sup\{|c_0|, |c_1|r, |c_2|r^2, \ldots \}. \]
We have defined the $M_i$ in such a way that the absolute value of the $k$th derivatives of $\Lambda$ and $\Lambda\Theta'$ at $t_0$ are bounded above by $M_1 k! r^{-k}$ and $M_2 k! r^{-k}$ respectively. 

A key step within the proof of the Cauchy--Kovalevskaya Theorem is to observe that the function 
\[ m_i(\widetilde{x}) = \frac{M_ir}{r-(\widetilde{x}-t_0)} \] 
has derivatives
\[ \frac{\mathrm{d}^k m_i}{\mathrm{d}\widetilde{x}^k} (t_0) = M_i k! r^{-k} \] 
for $i=1,2$. By replacing $\Lambda$ and $\Lambda\Theta'$ with $m_1$ and $m_2$ in equation \eqref{FOS} we obtain the new system
\begin{subequations}\label{ME}
\begin{empheq}[left =\empheqlbrace]{align}
&\frac{\partial\widetilde{\phi}}{\partial\widetilde{y}} = \left(\frac{M_1r}{r - (\widetilde{x}-t_0)}\right)\frac{\partial\widetilde{\phi}}{\partial\widetilde{x}} + \frac{M_2r}{r-(\widetilde{x}-t_0)} \\ 
&\widetilde{\phi}(\widetilde{x},0) = 0.
\end{empheq}
\end{subequations}
Since 
\[ \left| \frac{\mathrm{d}^k\Lambda}{\mathrm{d}\widetilde{x}^k} (t_0)\right| \leq \frac{\mathrm{d}^km_1}{\mathrm{d}\widetilde{x}^k} (t_0) \qquad \text{and} \qquad \left| \frac{\mathrm{d}^k(\Lambda\Theta')}{\mathrm{d}\widetilde{x}^k} (t_0)\right| \leq \frac{\mathrm{d}^k m_2}{\mathrm{d}\widetilde{x}^k} (t_0), \]
it follows that if $\phi$ and $\widetilde{\phi}$ solve systems \eqref{CBS2} and \eqref{ME} respectively, then
\begin{equation}\label{BOUN}
\left| \frac{\partial^{k+l}\phi}{\partial \widetilde{x}^k \partial \widetilde{y}^l} (t_0,0) \right| \leq \frac{\partial^{k+l}\widetilde{\phi}}{\partial \widetilde{x}^k \partial \widetilde{y}^l} (t_0,0), 
\end{equation} 
for $k,l\geq0$. 
Using the method of characteristics, the system \eqref{ME} has an explicit solution of the form $\widetilde{\phi}(\widetilde{x},\widetilde{y}) = \frac{M_2}{M_1} V(\widetilde{x}-t_0,\widetilde{y})$ in a neighbourhood of $(t_0,0)$ where
\begin{align*}
V(\widetilde{x},\widetilde{y}) := r - \widetilde{x} - \sqrt{(r-\widetilde{x})^2 - 2M_1 r \widetilde{y}}.
\end{align*}
This solution is analytic at $(t_0,0)$ and so by inequality \eqref{BOUN} the function $\phi$ has a convergent Taylor series at $(t_0,0)$ that solves system \eqref{CBS2}. This is usually where the proof of the Cauchy--Kovalevskaya Theorem ends, but we continue as we wish to find where the Taylor series of $\phi$ converges. 

{We will now attempt to find where the Taylor series for $V$ based at $(0,0)$ converges absolutely. Let $r_1,r_2>0$ and use $D(0,r_1)$ to denote the disc in the complex plane $\mathbb{C}$ centred at the origin with radius $r_1$. From the theory of complex analysis on several variables \cite[\S2]{Horm}, if we can show that $V$ is complex analytic in each variable separately on $D(0,r_1)\times D(0,r_2)$, then the Taylor series of V based at $(0,0)$ converges absolutely on $D(0,r_1)\times D(0,r_2)$. We can use this result to find out where in $\mathbb{R}^2$ the Taylor series of $V$ at $(0,0)$ converges absolutely. 

It is enough to find where the Taylor series of $\sqrt{(r-\widetilde{x})^2 - 2M_1 r \widetilde{y}}$ converges absolutely since it differs from $V$ by a linear term. Note that the square root function can be extended to $\mathbb{C}$ whilst being complex analytic away from the negative real axis. Take $b\in \mathbb{R}$ with $b\leq 0$ to be a point on the negative real axis. Let $a \in (0,r)$ and $(z_1,z_2)\in D(0,a) \times D(0,(r-a)^2/2M_1r)$. We plan to show $(r-z_1)^2-2M_1 rz_2$ remains away from the negative real axis so that $\sqrt{(r-z_1)^2-2M_1 rz_2}$ is complex analytic in each variable separately on $D(0,a) \times D(0,(r-a)^2/2M_1r)$. We do this by considering two cases. Firstly, if $(r-\RE(z_1))^2 \geq \IM(z_1)^2$, then $\RE( (r-z_1)^2)\geq 0$ and so we have
\begin{align*}
|(r-z_1)^2-2M_1 rz_2 -b| & \geq |(r-z_1)^2-b| - 2M_1 r|z_2| \\
& \geq |r-z_1|^2 - 2M_1 r|z_2| \\
& > (r-a)^2 -(r-a)^2 \\
&=0.
\end{align*}
Secondly, if instead $(r-\RE(z_1))^2 < \IM(z_1)^2$, then  
\begin{align*}
|(r-z_1)^2-2M_1 rz_2 -b| &\geq |\IM((r-z_1)^2-2M_1 rz_2)| \\
&= |2(r-\RE(z_1))\IM(z_1) -2M_1 r\IM(z_2)| \\
&\geq 2(r-\RE(z_1))|\IM(z_1)| -2M_1 r|\IM(z_2)| \\
& > 2(r-\RE(z_1))^2 - (r-a)^2 \\
& > (r-a)^2 \\
& > 0.
\end{align*}
Altogether this implies the Taylor series of $\sqrt{(r-z_1)^2-2M_1 rz_2}$ at $(0,0)$ converges absolutely on $D(0,a) \times D(0,(r-a)^2/2M_1r)$.

Thus the Taylor series for $V$ at $(0,0)$ converges absolutely on $D(0,a) \times D(0,(r-a)^2/2M_1 r)$. Therefore, the Taylor series of $\widetilde{\phi}$ at $(t_0,0)$ converges absolutely in the rectangle
\[ \left\{ (\widetilde{x},\widetilde{y}) \in\mathbb{R}^2: |\widetilde{x}-t_0|<a , \ |\widetilde{y}|<\frac{(r-a)^2}{2M_1r}\right\}. \]
We can take the union of these rectangles over $a\in(0,r)$ to obtain convergence within
\[ \widetilde{P_r}:= \left\{(\widetilde{x},\widetilde{y})\in \mathbb{R}^2 : |\widetilde{x}-t_0|<r,\  |\widetilde{y}| < \frac{(r-|\widetilde{x}-t_0|)^2}{2M_1r} \right\}. \]
Since this holds for all $r\in (0,\min\{R_1,R_2\})$ we can also take the union over $r$ to obtain convergence within 
\[ \widetilde{P} := \bigcup_{r\in (0,\min\{R_1,R_2\})} \widetilde{P_r}, \] 
which by inequalities \eqref{BOUN} implies system \eqref{CBS2} has a solution $\phi$ on $\widetilde{P}$. 

Before we change our variables back to $(x,y)$, we are interested in finding how far $\widetilde{P}$ extends in the $\widetilde{y}$ direction from $(t_0,0)$. Note that $(t_0,\widetilde{y})\in \widetilde{P_r}$ for $|\widetilde{y}|<\frac{r}{2M_1}$ and so since $M_1$ is dependent on $r$, the quantity of interest is $\sup_{r\in (0,\min\{R_1,R_2\})}\frac{r}{2M_1(r)}$. This quantity can be expressed as follows.
\begin{lem}\label{R0LEM}
By defining
\[ r_0:= \min\left\{ \frac{1}{\sup_{n\geq1}|b_n|^{\frac{1}{n}}}, R_2 \right\}, \]
it holds that
\begin{equation}\label{DIST}
\sup_{r\in(0,\min\{R_1,R_2\})}\frac{r}{2M_1(r)} = \frac{r_0}{2}. 
\end{equation} 
\end{lem}
\begin{proof}
We begin by showing $0<r_0 \leq \min\{R_1,R_2\}$. The inequality $r_0 \leq \min\{R_1,R_2\}$ is a consequence of
\[  \frac{1}{\sup_{n\geq 1}|b_n|^{\frac{1}{n}}} \leq  \frac{1}{\limsup_{n\to\infty}|b_n|^{\frac{1}{n}}} = R_1. \]
For $r\in(0,R_1)$ there exists some constant $C>1$ such that
\[ |b_n|r^n \leq \sum_{n=0}^\infty |b_n|r^n \leq C, \]
which implies $\sup_{n\geq1} |b_n|^{\frac{1}{n}} \leq \sup_{n\geq1}C^{\frac{1}{n}}/r \leq C/r<\infty$. It follows that $r_0>0$. 

To prove the equality \eqref{DIST} it is enough to show that 
\[ \inf_{r\in(0,\min\{R_1,R_2\})} \frac{M_1(r)}{r} = \frac{1}{r_0}. \]
By recalling the definition of $M_1(r)$, we have
\[ \frac{M_1(r)}{r} = \sup \left\{ \frac{1}{r}, |b_1|, |b_2|r ,|b_3|r^2,\ldots \right\}. \]
If $r\in (0,r_0]$, then $r \leq 1/\sup_{n\geq 1}|b_n|^{\frac{1}{n}}$ which implies $|b_n|r^{n-1} \leq 1/r$ for all $n\geq 1$. Therefore $M_1(r)/r = 1/r$ and so
\[ \frac{1}{r_0} = \inf_{r\in(0,r_0)}\frac{1}{r} = \inf_{r\in(0,r_0)}\frac{M_1(r)}{r} \geq \inf_{r\in(0,\min\{R_1,R_2\})} \frac{M_1(r)}{r}. \]

It remains to prove
\[ \frac{1}{r_0} \leq \inf_{r\in(0,\min\{R_1,R_2\})} \frac{M_1(r)}{r}, \]
which is equivalent to showing $1/r_0 \leq M_1(r)/r$ for all $r\in(0,\min\{R_1,R_2\})$. We have already shown $1/r_0\leq 1/r = M_1(r)/r$ for $r\in (0,r_0]$, so we only need to consider $r\in (r_0,\min\{R_1,R_2\})$. Of course if $r_0 =\min\{R_1,R_2\}$, then we are done as no such $r$ exist. Therefore, we assume $r_0 <\min\{R_1,R_2\}$, which must mean that $r_0 = 1/\sup_{n>0}|b_n|^{\frac{1}{n}}$

Suppose for a contradiction that there exists some $r\in(r_0,\min\{R_1,R_2\})$ such that $1/r_0>M_1(r)/r$. As a result, there exists an $\varepsilon\in(0,r-r_0)$ such that $ 1/(r_0 + \varepsilon)>|b_n|r^{n-1}$ for all $n\geq 1$. However, $r_0+\varepsilon >r_0 = 1/\sup_{n>0}|b_n|^{\frac{1}{n}}$ which implies there exists $n\geq 1$ such that 
\[ 1/(r_0 + \varepsilon)<|b_n|(r_0+\varepsilon)^{n-1}<|b_n|r^{n-1}.\] 
This provides us with a contradiction and proves that $1/r_0 \leq M_1(r)/r$ for all $r\in(r_0,\min\{R_1,R_2\})$ which concludes the proof.
\end{proof}

We are now ready to change back to our original variables $(x,y)$. We are only interested in the points that lie above the $\widetilde{x}$-axis and that are within the region where our change of variables is a diffeomorphism, so we define $\widetilde{Q} := \widetilde{P} \cap( (t_0-\delta,t_0+\delta)\times[0,\infty))$. When changing back to the $(x,y)$ variables, the region $\widetilde{Q}$ is mapped to a region $Q$ as depicted in the following figure. 
\begin{figure}[H]
\begin{tikzpicture}[scale=2]

\draw [fill=gray!20, thick, dashed] plot [smooth, samples=100, domain= -1:1] (\x+1.2, {(1-abs(\x))^2}) -- cycle;
\draw [thick] (0,0) -- (2.4,0);
\draw [fill] (1.2,0) circle [radius=1pt] node [below] {$(t_0,0)$};
\node at (1.2,0.4) {$\widetilde{Q}$};
\draw [fill] (1.2,1) circle [radius=1pt] node [right] {$(t_0,\frac{r_0}{2})$};

\node [above] at (3,0.5) {$(x,y) = \bm{\gamma}(\widetilde{x}) + \widetilde{y}\bm{\gamma}'^\perp(t_0)$};
\draw [thick, ->] (2.5,0.5) -- (3.5,0.5);

\begin{scope}[rotate around={-40:(5,0)}]
\draw [fill=gray!20, thick, dashed] plot [smooth, samples=100, domain=-1:1] (\x+5, {(1-abs(\x))^2+(\x)^3}) -- plot [smooth, samples=100, domain=1.1:-1.1] (\x+5, {(\x)^3});
\draw [thick] plot [smooth, samples=100, domain=-1.1:1.1] (\x+5, {(\x)^3}) node [right] {$\bm{\gamma}$};
\draw [fill] (5,0) circle [radius=1pt] node [below left] {$\bm{\gamma}(t_0)$};
\node at (5,0.4) {$Q$};
\draw [fill] (5,1) circle [radius=1pt] node [right] {$\bm{\gamma}(t_0) + \frac{r_0}{2}\bm{\gamma}'^\perp(t_0)$};
\end{scope}

\end{tikzpicture}
\centering
\caption{Mapping $\widetilde{Q}$ to $Q$ via the change of variables from $(\widetilde{x},\widetilde{y})$ to $(x,y)$.}
\label{COV}
\end{figure}
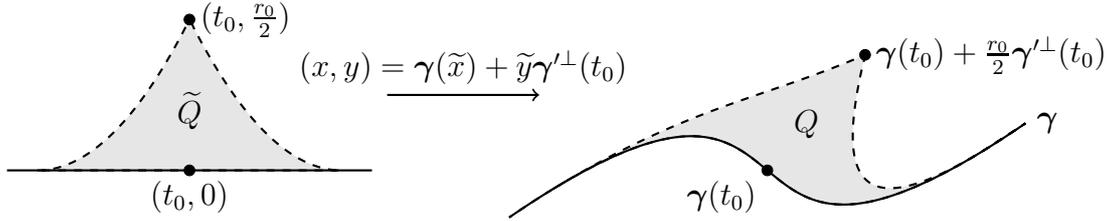
Altogether, we have shown system \eqref{SYS4} has a solution $\mathcal{B}$ on $Q$. It is clear from Figure \ref{COV} and Lemma \ref{R0LEM} that we can find an exterior collar neighbourhood within $Q$ that has the form
\[ N_{t_0} :=  \bigcup_{t\in(t_0-\delta,t_0+\delta)} \{\bm{\gamma}(t) + \varepsilon\bm{\gamma}^\perp (t) : \varepsilon\in[0,l(t)) \}, \]
with $l\colon (t_0-\delta,t_0+\delta)\to (0,\infty]$ satisfying $l(t_0) =\frac{r_0}{2}$. By recalling the definition of $r_0$ and that the $b_n$ and $R_2$ are dependent on $t_0$, we can express the width of $N_{t_0}$ at $\bm{\gamma}(t_0)$ as
\[ d(t_0) := \frac{r_0}{2} |\bm{\gamma}(t_0)^\perp| =  \frac{|\bm{\gamma}'(t_0)|}{2} \min\left\{ \frac{1}{\sup_{n\geq1}|b_n(t_0)|^{\frac{1}{n}}}, R_2(t_0)\right\}. \]

Now by letting $N^*$ be an exterior collar neighbourhood of $\partial\Omega$ with constant width $l^*$, we can patch together the exterior collar neighbourhoods $N_{t}\cap N^*$ over $t\in\mathbb{T}$ in the same way as in the end of the proof of Theorem \ref{CPHV}. This generates a solution to the Cauchy problem \eqref{ECP} on the exterior collar neighbourhood of constant width
\[ d^* := \min \left\{ \inf_{t\in\mathbb{T}} d(t), l^* \right\},  \]
where we take $l^*$ to be the maximum width of all possible constant width exterior collar neighbourhoods of $\partial\Omega$. Altogether, we have proven the following theorem. 
\begin{thm}
For $f,h\in C^\omega(\partial\Omega)$ there exists $\B \in C^1(U\setminus \overline{\Omega};\mathbb{R}^2) \cap C(U\setminus \Omega;\mathbb{R}^2)$ that solves the Cauchy problem \eqref{ECP} on the exterior collar neighbourhood of constant width $d^*$ which has the form
\[ U\setminus \Omega = \bigcup_{\bm{w}\in\partial\Omega} \{\bm{w}+ \varepsilon\N(\bm{w}) : \varepsilon\in[0,d^*)\}. \]
\end{thm}

We have shown that there exists an external harmonic extension to at least a distance $d^*$ away from $\partial\Omega$. We now remark on how $d^*$ relates to the curvature of $\partial\Omega$.
\begin{rmk}
Observe that 
\[ \RE \Lambda =  \frac{\gamma_1'(t_0)\gamma_2'-\gamma_2'(t_0)\gamma_1'}{|\bm{\gamma}'|^2}, \]
and hence
\[ (\RE \Lambda)'(t_0) =  \frac{\gamma_1'(t_0)\gamma_2''(t_0)-\gamma_2'(t_0)\gamma_1''(t_0)}{|\bm{\gamma}'(t_0)|^2}. \]
The curvature of $\bm{\gamma}$ is 
\[ \kappa = \frac{|\gamma_1'\gamma_2''-\gamma_2'\gamma_1''|}{|\bm{\gamma}'|^3} \]
which implies 
\[ \kappa(t_0) = \frac{|(\RE \Lambda)'(t_0)|}{|\bm{\gamma}'(t_0)|} = \frac{|\RE b_1(t_0)|}{|\bm{\gamma}'(t_0)|} \] 
Therefore, 
\[ d^*\leq d(t_0) \leq \frac{|\bm{\gamma}'(t_0)|}{2|b_1(t_0)|} \leq \frac{|\bm{\gamma}'(t_0)|}{2|\RE b_1(t_0)|} = \frac{1}{2\kappa(t_0)} \] 
and so
\[ d^* \leq \frac{1}{2}\inf_{\mathbb{T}}\left(\frac{1}{\kappa}\right). \]
This shows that our lower bound on how far we can harmonically extend is no more than half the minimum radius of curvature.
\end{rmk}
We now go about finding $d^*$ for some simple examples where we can compute the quantity $\sup_{n\geq 1}|b_n(t_0)|^{\frac{1}{n}}$ explicitly. Note that two of our examples are for boundaries $\partial\Omega$ that are not closed curves, however, our workings can easily be adapted to these settings.
\begin{eg}
Let $\partial\Omega$ be the circle of radius $R>0$ parameterised clockwise by $\bm{\gamma}(t) = R(\cos t,-\sin t)$, and suppose the boundary data $f(\bm{\gamma}(t)),h(\bm{\gamma}(t))$ has an analytic continuation to $\mathbb{C}$. We have
\[ \Lambda(\widetilde{x}) = ie^{i(\widetilde{x}-t_0)} = \sum_{n=0}^\infty \frac{i^{(n+1)}}{n!}(\widetilde{x}-t_0)^n. \]
It follows that $|b_n(t_0)| = 1/n!$, which implies $\sup_{n\geq 1}|b_n(t_0)|^{\frac{1}{n}}=1$. Both $\Lambda$ and $\Theta$ have an analytic continuation to $\mathbb{C}$ and so $R_2(t)=\infty$ for all $t\in\mathbb{T}$. Therefore, $d(t) = R/2$. In the case of a circle we have $l^*=\infty$. Overall, $d^* = R/2$, which shows that in this setting $d^*$ is dependent on the curvature of $\partial\Omega$.
\end{eg}

\begin{eg}
Let $\partial\Omega$ be the flat boundary of the form $\bm{\gamma}(t)=(t,0)$, and suppose the boundary data $f(\bm{\gamma}(t)),h(\bm{\gamma}(t))$ has an analytic continuation to the complex strip $\{z\in\mathbb{C} : |\IM(z)|<a\}$ for some $a>0$ and no further. We have $\Lambda(\widetilde{x}) = i$, which implies $|b_n(t_0)|=0$ for $n\geq 1$ and $\sup_{n\geq 1}|b_n(t_0)|^{\frac{1}{n}}=0$. The function $\Lambda\Theta'$ has an analytic continuation to $\{z\in\mathbb{C} : |\IM(z)|<a\}$ and no further, which implies $\inf_{t\in\mathbb{R}} R_2(t)= a$. Hence $\inf_{t\in\mathbb{R}} d(t)= a/2$. In the case of a flat boundary $l^*=\infty$ and thus $d^*= a/2$, which shows that in this setting $d^*$ is dependent on the extent to which the boundary data can be analytically continued. 
\end{eg}

\begin{eg}
Let $\partial\Omega$ be the parabola of the form $\bm{\gamma}(t)=(t,t^2)$, and suppose the boundary data is such that $f(\bm{\gamma}(t)),h(\bm{\gamma}(t))$ has an analytic continuation to $\mathbb{C}$. We have
\[ \Lambda(\widetilde{x}) = i\left( \frac{1+2t_0i}{1+2\widetilde{x}i}\right) = \sum_{n=0}^\infty i\left(\frac{-2i}{1+2t_0i}\right)^n (\widetilde{x}-t_0)^n, \]
which implies $|b_n(t_0)|^{\frac{1}{n}}= 2/\sqrt{1+4t_0^2}$ and $\sup_{n\geq1}|b_n(t_0)|^{\frac{1}{n}}= 2/\sqrt{1+4t_0^2}$. The functions $\Lambda(t)$ and $|\bm{\gamma}'(t)| = \sqrt{1+4t^2}$ both have an analytic continuation to the complex strip $\{z\in\mathbb{C} : |\IM(z)|<1/2 \}$, which implies $\Lambda\Theta'$ also has an analytic continuation there. Consequently, $R_2(t)\geq 1/2$ for all $t\in\mathbb{R}$. It follows that $\inf_{t\in\mathbb{R}} d(t) = d(0) = 1/4$ since $1/\sup_{n\geq1}|b_n(t)|^{\frac{1}{n}}\geq 1/2$ and $1/\sup_{n\geq1}|b_n(0)|^{\frac{1}{n}}=1/2$. The quantity $l^*$ for this parabola is the smallest radius of curvature of $\bm{\gamma}$, which turns out to be $1/2$. We therefore conclude $d^*=1/4$. 
\end{eg}

In the case where $\sup_{n\geq 1}|b_n(t_0)|^{\frac{1}{n}}$ can not computed explicitly, we can approximate it using the following method. For $t_0\in\mathbb{T}$, let $a_n(t_0)\in\mathbb{C}$ be the Taylor coefficients satisfying
\[ \gamma_1'(\widetilde{x})+i\gamma_2'(\widetilde{x}) = \sum_{n=0}^\infty a_n(t_0) (\widetilde{x}-t_0)^n. \]
Then the Taylor coefficients of $\Lambda$ at $t_0$ can be expressed in terms of the $a_n$ as
\[ b_n = \frac{i}{(\gamma_1'(t_0)+i\gamma_2'(t_0))^n} \det A_n \]
for $n\geq 1$ where 
\[ A_n = \begin{pmatrix}
0 			& a_1 		& a_2 		& \cdots & a_n 		\\
0 			& a_0 		& a_1 		& \cdots & a_{n-1} 	\\
0 			& 0    		& a_0 		& \cdots & a_{n-2} 	\\
\vdots 	& \vdots	& \vdots	& \ddots & \vdots	\\
1			& 0			& 0			& \cdots & a_0										
\end{pmatrix}. \]
To prove this it is enough to show that
\begin{equation}\label{TSM}
 b_n = -\frac{1}{a_0} \sum_{k=1}^n a_k b_{n-k}, 
\end{equation}
which comes from multiplying the Taylor series of $\gamma_1'+i\gamma_2'$ and $\Lambda$. To show that the expression for $b_n$ satisfies \eqref{TSM}, expand $\det A_n$ by the first row and then keep expanding the determinants of the minors by the columns consisting only of $a_0$ until the result is obtained. 

We bound the expression for the Taylor coefficients $b_n$ using Hadamard's Inequality \cite[\S14.1]{Gar}.
\begin{thm}[Hadamard's Inequality]
Let $M=(m_{j,k})$ be a real or complex $n\times n$ matrix. Then
\[ |\det M| \leq \prod_{k=1}^n \left( \sum_{j=1}^n |m_{j,k}|^2 \right)^{\frac{1}{2}}. \]
\end{thm}
Using this inequality we have
\begin{align*}
|b_n| &= \frac{1}{|\bm{\gamma}'(t_0)|^n} |\det A_n | \\
&\leq \frac{1}{|\bm{\gamma}'(t_0)|^n} \prod_{k=1}^n \left( \sum_{j=0}^{k} |a_j|^2 \right)^{\frac{1}{2}} \\
&\leq \frac{1}{|\bm{\gamma}'(t_0)|^n}\left( \sum_{j=0}^{n} |a_j|^2 \right)^{\frac{n}{2}} \\
&\leq \frac{1}{|\bm{\gamma}'(t_0)|^n}\left( \sum_{j=0}^{n} |a_j|\right)^n,
\end{align*} 
which implies
\begin{equation}\label{APROX}
\sup_{n\geq 1} |b_n(t_0)|^{\frac{1}{n}} \leq \frac{1}{|\bm{\gamma}'(t_0)|} \sum_{n=0}^{\infty} |a_n(t_0)|. 
\end{equation} 
We can use this result to find an approximate of $d^*$ for more complicated boundaries as shown in the following example.
\begin{eg}
Suppose the boundary $\partial\Omega$ can be parameterised by $\bm{\gamma}$ that has the form of the finite Fourier series 
\[ \gamma_1(t)+i\gamma_2(t) = \sum_{k=-N}^N c_k e^{ikt} \]
for some $N\geq 1$ and $c_k\in\mathbb{C}$. Thus
\[ a_n(t_0) = \sum_{k=-N}^N \frac{c_ke^{ikt_0}(ik)^{n+1}}{n!} \]
and so by substituting this into inequality \eqref{APROX}, we obtain 
\[ \sup_{n\geq 1} |b_n(t_0)|^{\frac{1}{n}} \leq \frac{1}{|\bm{\gamma}'(t_0)|} \sum_{n=0}^{\infty} \sum_{k=-N}^N \frac{|c_k| |k|^{n+1}}{n!} = \frac{1}{|\bm{\gamma}'(t_0)|} \sum_{k=-N}^N |c_k||k|e^{|k|}. \]
If we suppose the boundary data is such that $f(\bm{\gamma}(t))/|\bm{\gamma}'(t)|$ and $h(\bm{\gamma}(t))/|\bm{\gamma}'(t)|$ have an analytic continuation to $\mathbb{C}$, then $R_2(t)=\infty$ for all $t\in\mathbb{T}$. Therefore,
\[ d(t) \geq \frac{|\bm{\gamma}'(t)|^2}{2\sum_{k=-N}^N |c_k||k|e^{|k|}} \]
which provides us with
\[ d^* \geq \min \left\{ \frac{\inf_{t\in\mathbb{T}}|\bm{\gamma}'(t)|^2}{2\sum_{k=-N}^N |c_k||k|e^{|k|}} , l^* \right\}. \]
The right hand side is a lower bound for $d^*$ and therefore a lower bound on how far we can harmonically extend from a boundary that is represented by a finite Fourier series.
\end{eg}

\section{Acknowledgements}
We would like to thank Per Helander and Robert MacKay for providing direction. This work was supported by a PhD studentship from the Simons Foundation (601970, RSM).  

\bibliographystyle{unsrt}
\bibliography{2DHarmonicExtensions}

\end{document}